\documentclass[a4paper,11pt]{article}
\usepackage{stmaryrd}
\usepackage{amsfonts}
\usepackage{bbm}
\usepackage{amscd}
\usepackage{mathrsfs}
\usepackage{latexsym,amssymb,amsmath,amscd,amscd,amsthm,amsxtra,xypic}
\usepackage[dvips]{graphicx}
\usepackage[utf8]{inputenc}
\usepackage[T1]{fontenc}
\usepackage{enumerate}
\usepackage{lmodern}
\usepackage{amssymb}
\usepackage[all]{xy}
\usepackage{ifpdf}
\ifpdf
\usepackage[colorlinks=true,linkcolor=blue,final,backref=page,hyperindex,citecolor=red]{hyperref}
\else
\usepackage[colorlinks,final,backref=page,hyperindex,hypertex]{hyperref}
\fi

\usepackage{nicefrac,mathtools,enumitem}
\usepackage{microtype}
\usepackage{amsfonts}
\usepackage{amssymb}
\usepackage{mathrsfs}
\usepackage{amsmath}
\usepackage{amsthm}
\usepackage{enumerate}
\usepackage{indentfirst}
\usepackage{amsfonts}
\usepackage{amssymb}
\usepackage{mathrsfs}
\usepackage{amsmath}
\usepackage{amsthm}
\usepackage{enumerate}
\usepackage{cite}
\usepackage{mathrsfs}
\usepackage{kpfonts}
\usepackage{geometry}
\allowdisplaybreaks
\usepackage{cancel}

\def\<{\langle}
\def\>{\rangle}

\def\c{\cdot}

\newtheorem{thm}{Theorem}[section]

\newtheorem{cor}[thm]{Corollary}
\newtheorem{pro}[thm]{Proposition}
\newtheorem{ex}[thm]{Example}

\theoremstyle{definition}
\newtheorem{defi}{Definition}[section]

\theoremstyle{remark}
\newtheorem{rmk}{Remark}[section]
\begin{document}
\title{ \bf
Maurer-Cartan  type  cohomology on generalized Reynolds operators and  NS-structures on Lie triple systems}
\author{\bf R.Gharbi, S. Mabrouk, A. Makhlouf}
\author{{Rahma Gharbi$^{1,3}$
 \footnote { E-mail: Rahma.gharbi@uha.fr}
,\  Sami Mabrouk$^{2}$
 \footnote { E-mail: mabrouksami00@yahoo.fr }
\ and Abdenacer Makhlouf$^{1}$
 \footnote { E-mail: abdenacer.makhlouf@uha.fr $($Corresponding author$)$}
}\\
{\small 1. ~ IRIMAS - Département de Mathématiques, 18, rue des frères Lumière,
F-68093 Mulhouse, France} \\
{\small 2.  University of Gafsa, Faculty of Sciences, 2112 Gafsa, Tunisia}\\
{\small 3.  University of Sfax, Faculty of Sciences Sfax,  BP
1171, 3038 Sfax, Tunisia}\\
}
\date{}
\maketitle
\begin{abstract}
The purpose  of this paper is to introduce and study the notion of   generalized Reynolds operators on Lie triple systems with representations (Abbr.  \textsf{L.t.sRep}  pairs) as generalization of  weighted Reynolds operators on Lie triple systems. First, We construct an $L_{\infty}$-algebra  whose Maurer-Cartan elements  are generalized Reynolds  operators. This allows us to define a Yamaguti cohomology of a generalized Reynolds operator. This cohomology can be seen as the Yamaguti cohomology of a certain Lie triple system  with coefficients in a
suitable representation. Next, we study  deformations of generalized Reynolds operators from cohomological points of view and we investigate  the
obstruction class of an extendable deformation of order $n$.  We end this paper by introducing a new algebraic structure, in connection  with  generalized Reynolds operator, called NS-Lie triple system. Moreover, we show that NS-Lie triple systems can be derived from NS-Lie algebras.
\end{abstract}


$L_{\infty}$-algebra, Lie-Yamaguti cohomology, deformation,  NS-Lie triple system.


\numberwithin{equation}{section}

\tableofcontents

\section{Introduction}

The concept of Lie triple system   was introduced  first  by  Jacobson \cite{Jacobson} and the present formulation is due to Yamguti \cite{Yamaguti}. Moreover, it appeared in  Cartan's work on Riemannian Geometry  \cite{Cartan} and was strongly developed for Symmetric spaces and related spaces. Indeed,  the tangent space of a symmetric space is a Lie triple system. It turns out that  they have important applications in physics, in particular, in elementary particle theory and the theory of quantum mechanics, as well as numerical analysis of differential equations. They  have become an interesting subject in mathematics, their structure have been studied first by  Lister in \cite{Lister}.

The notion of a Rota-Baxter operator on an associative algebra appeared first in  probability theory when  G. Baxter studied algebraically the Spitzer identity. The theory was developed first by Rota in the combinatorics field.  Later  relevant application was considered  in the Connes-Kreimer's algebraic approach to renormalization of quantum field theory \cite{CK}. For further details, see ~\cite{Gub-AMS,Gub}. A Rota-Baxter operator on a Lie algebra is naturally the operator form of a classical $r$-matrix~\cite{STS} under certain conditions. To better understand such connection in general, Kupershmidt introduced the notion of an $\mathcal{O}$-operator (also called
a relative Rota-Baxter operator \cite{PBG} or a generalized Rota-Baxter operator \cite{Uch})
on a Lie algebra in~\cite{Ku}. Recently a   deformation theory and a cohomology theory of relative Rota-Baxter operators on both Lie and associative algebras are studied  in \cite{Das,TBGS}. Reynolds operators were introduced by Reynolds in \cite{Re} in the study of fluctuation theory in fluid dynamics.
In \cite{KAM}, the author coined the concept of the Reynolds operator and regarded the operator as a mathematical subject in general. 
In \cite{gao-guo}, the authors provided examples and properties of the Reynolds operators and studied the free Reynolds algebras.

 NS-algebras were introduced by Leroux in \cite{Leroux}. In \cite{LG}, the authors studied   the relationship between the category of Nijenhuis algebras and the category of NS-algebras. Uchino defined the notion of generalized Reynolds operators on associative algebras and observed that a generalized Reynolds operator induces an NS-algebra in \cite{Uch}.  In \cite{Das-2}, Das defined a  cohomology of generalized Reynolds operators on associative algebras and a cohomology of NS-algebras,  various applications were provided.  He  also introduced the notions of generalized Reynolds operators, Reynolds operators on Lie algebras and NS-Lie algebras, and showed that NS-Lie algebras are the underlying algebraic structures of generalized Reynolds operators on Lie algebras  in \cite{Das-1}.

Inspired by these works, we aim in this paper to study Maurer-Cartan  type  cohomology on generalized Reynolds operators on Lie triple systems with representations (Abbr.  \textsf{L.t.sRep}  pairs).  Moreover, we introduce the notion of NS-Lie triple system as the underlying structure of generalized Reynolds operators.

 This is paper is organized as follow, in Section \ref{Section-2}, we summarize some basics and briefly recall representations and cohomology
of Lie triple systems. In Section \ref{Section-3}, we introduce  a generalization of  weighted Reynolds operators on a Lie triple system called generalized Reynolds operators on a Lie triple system and we  give some constructions. In Section \ref{Section-4},  we construct an $L_\infty$-algebra whose Maurer–Cartan elements are given by  generalized Reynolds operators. This motivates us to
define  a Yamaguti cohomology of  generalized Reynolds operator on Lie triple system using the underlying Lie triple systems  of the generalized Reynolds operator. In Section \ref{Section-5}, we  describe the obstruction class of a deformation of order $n$ to be extended to order $n+1$.  Finally,  Section  \ref{Section-6} deals   a new algebraic structure called NS-Lie triple system that is related to generalized Reynolds operators. We show that  Lie
algebras, NS-Lie algebras, Lie triple systems and
NS-Lie triple systems are closely related.

In this paper all vector spaces are considered over a field $\mathbb{K}$  of characteristic $0$.
\section{Preliminaries}\label{Section-2}
In this section, we recall some basic notions as  representations and cohomology theory  of Lie triple systems (L.t.s) introduced in \cite{Lister}. A  Lie triple system   is a vector space $L$ endowed with a ternary bracket $[\cdot,\cdot,\cdot] : \wedge^2L \otimes L \to L$ satisfying
\begin{align}
    & [x,y,z]+[y,z,x]+[z,x,y]=0, \label{lts 1}
    \\
    & [x,y,[z,t,e]]=[[x,y,z],t,e]+[z,[x,y,t],e]+[z,t,[x,y,e]] \quad\label{lts 2}\forall  x,y,z,t,e \in L.
\end{align}

A morphism $\phi : (L, [\c, \c, \c]) \rightarrow (L^{'}
, [\c, \c, \c]^{'})$ of L.t.s is a linear map satisfying  : \begin{equation*}
   \phi([x, y, z]) = [\phi(x), \phi(y), \phi(z)]^{'}\quad \forall x,y,z\in L.
\end{equation*}
An isomorphism is a bijective morphism.
We set  $X = (x, y)\in L\otimes L$ and denote by  $\mathfrak R(X),\ \mathfrak D(X) :L \rightarrow L$ maps defined by $\mathfrak R(X)z  = [ z,x, y]$ and  $\mathfrak D(X)z = [x, y, z]$. Then the identities  \eqref{lts 1} and \eqref{lts 2}  can be rewritten respectively
in the form $\mathfrak D(X)=\mathfrak R(\tau(X))-\mathfrak R(X)$, where $\tau(X)=y\otimes x$, and
\begin{equation*}
    \mathfrak D(X)[z_1, z_2, z_3] = [\mathfrak D(X)z_1, z_2, z_3] + [z_1, \mathfrak D(X)z_2, z_3] + [z_1, z_2, \mathfrak D(X)z_3].
\end{equation*}
This means that $\mathfrak D(X)$ is a derivation with respect to the bracket $[\c, \c,\c ]$.

In \cite{Yamaguti},  K. Yamaguti introduced the notion of representation and cohomology theory of L.t.s. Later, the authors in \cite{Harris,Hodge,Kubo} studied the cohomology theory of L.t.s from a different point of view. K. Yamaguti’s
work can be described as follows.

A representation of a L.t.s $L$ on a vector space $M$ is a bilinear map $\theta : \otimes^{2} L \rightarrow End(M)$, such that the following conditions are satisfied  :
\begin{align}
    & \theta(z,t)\theta(x,y)-\theta(y,t)\theta(x,z)-\theta(x,[y,z,t])+D(y,z)\theta(x,t)=0, \label{rep lts 1}
    \\
    &\theta(z,t)D(x,y)-D(x,y)\theta(z,t)+\theta([x,y,z],t)+\theta(z,[x,y,t])=0 \quad  \forall x,y,z,t\in L,  \label{rep lts 2}
\end{align}
where $D(x,y)=\theta(y,x)-\theta(x,y)$. It is  denoted  by a pair $(M,\theta)$. We say that we have a \textsf{L.t.sRep} pair
and refer to it with the tuple
$(L, [\cdot,\cdot, \cdot],\theta)$.
\begin{rmk}
\begin{enumerate}
    \item
    It can be concluded from \eqref{rep lts 2} that
\begin{align}\label{rep lts 3}
    &D(z,t)D(x,y)-D(x,y)D(z,t) + D([x,y,z],t) + D(z,[x,y,t]) =0.
\end{align}

\item If $M=L$  then $\mathfrak R$ is a representation of $L$ in it self, which is called the regular representation.\end{enumerate}\end{rmk}
 Recall that, a  representation of a Lie algebra $(L,[\c,\c] )$  on  a vector space $M$ is  a  linear map $\rho : L \to End( M) $ satisfying  :
\begin{align*}
    &\rho([x,y]) = \rho(x)\rho(y)-\rho(y)\rho(x)\quad\forall x,y \in L.
\end{align*}
Note that a \textsf{LieRep} pair $(L,[\c,\c],\rho)$ is a Lie algebra  $(L,[\c,\c] )$ with a representation $\rho$.
\begin{thm}\label{Lie.L.t.s}
Let $(L,[\c,\c],\rho)$ be a \textsf{LieRep} pair.
Then $(L, [\c,\c,\c],\theta_{\rho})$ is a \textsf{L.t.sRep} pair, where  \begin{equation}
    [x,y,z]=[[x,y],z]
\quad \text{and}\quad
    \theta_{\rho}(x,y)=\rho(y)\rho(x)\quad \forall  x,y,z\in L.
\end{equation}
\end{thm}
Let
$(L, [\cdot,\cdot, \cdot],\theta)$ be a \textsf{L.t.sRep} pair. For each $n \geqslant 0$, we denote by $\mathcal{C}_{L.t.s}^{2n+1}(L,M)$, the vector space of $(2n+1)$-cochains of $L$ with coefficients in $M$ : $\psi\in \mathcal{C}_{L.t.s}^{2n+1}(L,M)$ is a multilinear function of $\psi :\times^{2n+1}L \to M$ satisfying
\begin{align*}
&\quad\quad\quad\psi(x_1,x_2,\cdots,x_{2n-2},x,x,y)=0,\\
&\circlearrowleft_{x, y, z}\psi(x_1, x_2, \cdots, x_{2n-2}, x, y, z)  = 0,
\end{align*}
where $\circlearrowleft_{x, y, z}$ means that we are taking a cyclic summation.

A Yamaguti coboundary operator $\delta^{2n-1} : \mathcal{C}_{L.t.s}^{2n-1}(L,M) \rightarrow \mathcal{C}_{L.t.s}^{2n+1}(L,M) $ is defined by:
\begin{align}
  &  \delta^{2n-1} \psi(x_1,x_2, \cdots , x_{2n+1})\nonumber\\
    =&\theta(x_{2n},x_{2n+1})\psi(x_1,x_2, \cdots , x_{2n-1})- \theta(x_{2n-1},x_{2n+1})\psi(x_1,x_2, \cdots , x_{2n-2},x_{2n}) \nonumber \\
    & +\sum_{k=1}^n (-1)^{n+k}D(x_{2k-1},x_{2k})\psi(x_1,x_2, \cdots , \widehat{x}_{2k-1},\widehat{x}_{2k}, \cdots , x_{2n+1}) \nonumber \\
   &  +\sum_{k=1}^n \sum_{j=2k+1}^{2n+1}  (-1)^{n+k+1} \psi(x_1,x_2, \cdots,  \widehat{x}_{2k-1},\widehat{x}_{2k}, \cdots, [x_{2k-1},x_{2k},x_j],\cdots , x_{2n+1}),
\end{align}
for all $\psi \in  \mathcal{C}_{L.t.s}^{2n-1}(L,M), n \geqslant 1$, where $\;\widehat{}\;$ means that the element is omitted.
\begin{pro}
    The Yamaguti coboundary operator $\delta^{2n-1}$ defined above is a square zero map, i.e., $\delta^{2n+1}\circ \delta^{2n-1}=0$.
Therefore,     the  set of Yamaguti cochains forms a complex with this coboundary as follow:
$$ \mathcal{C}_{L.t.s}^{1}(L,M)\overset{\delta^{1}}{\longrightarrow} \mathcal{C}_{L.t.s}^{3}(L,M)\overset{\delta^{3}}{\longrightarrow}  \mathcal{C}_{L.t.s}^{5}(L,M) \longrightarrow \cdots.$$
\end{pro}
Hence we get the
Yamaguti cohomology group  $\mathcal{H}_{L.t.s}^{\bullet}(L , M ) = \mathcal{Z}_{L.t.s} ^{\bullet}(L , M )/\mathcal{B}_{L.t.s}^{\bullet}(L , M )$, where $\mathcal{Z}_{L.t.s}^{\bullet}(L ,M )$ is the space of
cocycles and $\mathcal{B}_{L.t.s}^{\bullet}(L , M )$ is the space of coboundaries. See \cite{Yamaguti,Zhang} for more details.
A  $1$-cochain $\varphi \in\mathcal{C}_{L.t.s}^{1}(L, M )$
is a $1$-cocycle if
\begin{equation}\label{1cocycle}
    D(x_1,x_2)\varphi (x_3)-\theta(x_1,x_3)\varphi (x_2)+\theta(x_2,x_3)\varphi(x_1)-\varphi ([x_1,x_2,x_3])=0.
\end{equation}
A  $3$-cochain $\mathcal{H}\in \mathcal{C}_{L.t.s}^{3}(L,M)$ is a $3$-coboundary if there exists a map $\varphi \in \mathcal{C}_{L.t.s}^{1}(L, M )$ such that $\mathcal{H}=\delta^{1}\varphi$.
 A $3$-coboundary  is called a $3$-cocycle if for all $ x_1,x_2,x_3,y_1,y_2,y_3\in L$ :
{\small\begin{align}
  \label{SkewCochain}  &\mathcal{H}(x_1,x_1,x_2)=0,\\\label{2.11}
    &\mathcal{H}(x_1,x_2,x_3)+\mathcal{H}(x_2,x_3,x_1)+\mathcal{H}(x_3,x_1,x_2)=0,\\
    &\mathcal{H}(x_1,x_2,[y_1,y_2,y_3])+D(x_1,x_2)\mathcal{H}(y_1,y_2,y_3)-\mathcal{H}([x_1,x_2,y_1],y_2,y_3)-\mathcal{H}(y_1,[x_1,x_2,y_2],y_3)\nonumber\\
   = &\mathcal{H}(y_1,y_2,[x_1,x_2,y_3])
    +\theta(y_2,y_3)\mathcal{H}(x_1,x_2,y_1)-\theta(y_1,y_3)\mathcal{H}(x_1,x_2,y_2)+D(y_1,y_2)\mathcal{H}(x_1,x_2,y_3).\label{2.12}
\end{align}}

 \begin{thm} \cite{O-operator}
Let $\phi\in \mathcal{Z}^{2}_{Lie}(L,M)$. Then $\omega(x,y,z)=\phi([x,y],z)-\rho(z)\phi(x,y)$ is a $3$-cocycle of the \textsf{L.t.sRep} pair $(L, [\c,\c,\c],\theta_{\rho})$ given in Theorem \ref{Lie.L.t.s}.
\end{thm}
\section{Generalized Reynolds operators  on Lie triple systems with representations}\label{Section-3}
In this section,  we introduce the notion of generalized Reynolds operator on \textsf{L.t.sRep} pairs which provides a generalization of  weighted Reynolds operators on L.t.s.   Also, We give its characterization   by a graph and provide some constructions. This notion is an algebraic analogue of     generalized Reynolds operator on a Lie algebra introduced in \cite{Das-Twisted} and on a $3$-Lie algebras introduced in \cite{chtioui,Sheng}.

\subsection{Weighted Reynolds operators  on L.t.s}
 \begin{defi}
  Let $(L,[\c,\c] )$ be a Lie algebra and  $\lambda$ be a non-null scalar . A \textbf{ $\lambda$-weighted Reynolds operator} on  a  Lie algebra  is a linear map  $\mathcal{R}  : L\to L $ satisfying:
\begin{align}
    &[\mathcal{R}x,\mathcal{R}y]=\mathcal{R}\Big (  [\mathcal{R}x,y]+  [x,\mathcal{R}y] +\lambda [\mathcal{R}x,\mathcal{R}y]\Big)  \quad \forall x,y \in L.
\end{align}
Note that $(L,[\c,\c],\mathcal{R} )$ is  $\lambda$-weighted Reynolds Lie algebra.
 \end{defi}
 \begin{rmk}
If $\lambda=-1$, we recover the definition of Reynolds operator on a Lie algebra introduced by A. Das  in \cite{Das-Twisted} in the study of generalized Reynolds operators on a Lie algebra.
 \end{rmk}
 \begin{defi}
  Let $(L,[\c,\c,\c] )$ be a  L.t.s and  $\lambda$ be a scalar. A \textbf{$\lambda $-weighted Reynolds operator} on a L.t.s  is a linear map  $\mathcal{R}  : L\to L $ satisfying:
  \begin{align}\label{Reynolds}
    &[\mathcal{R}x,\mathcal{R}y,\mathcal{R}z]=\mathcal{R}\Big (  [\mathcal{R}x,\mathcal{R}y,z]+  [x,\mathcal{R}y,,\mathcal{R}z] +  [\mathcal{R}x,y,\mathcal{R}z]+\lambda [\mathcal{R}x,\mathcal{R}y,\mathcal{R}z]\Big) \quad\forall x,y,z \in L.
\end{align}
Note that $(L,[\c,\c,\c],\mathcal{R} )$ is $\lambda $-weighted Reynolds  \textsf{L.t.s}.
  \end{defi}
\begin{ex}\label{Block}
For a fixed complex number $q$,
there is a  Lie triple system structure analogue to the Lie algebra of Block type given in \cite{Block}, generated by   the basis $\{L_{m,i} | m, i \in \mathbb{Z}\}$ and given by  the following bracket
$$[L_{m,i}, L_{n,j}, L_{p,k} ] = (n(i + q) - m(j + q)) (p(i + j+q) - (m+n)(p + q))L_{m+n+p,i+j+k}. $$
Let $B(q)_{\geq0}= \{L_{m,i} | m, i \in \mathbb{Z}_+\}$, then the linear map  $ \mathcal{R} :B(q)_{\geq0} \to B(q)_{\geq0} $ defined by \begin{align}
    \mathcal{R} (L_{m,i})= \frac{- 1}{\lambda(m+i+1)}\  L_{m,i} \ \ \ \ \forall m \in B(q)_{\geq0}, \lambda\neq 0
\end{align}
is a $\lambda $-weighted Reynolds operator on  $B(q)_{\geq0}$.

\end{ex}
  \begin{defi}
  Let $(L,[\c,\c,\c],\mathcal{R} )$  and $(L',[\c,\c,\c]',\mathcal{R'} )$ be two $\lambda $-weighted Reynolds  L.t.s. A linear map $\alpha  : L \to L'$ is called \textbf{morphism} of $\lambda $-weighted Reynolds  L.t.s  if $\alpha$ is morphism of  L.t.s satisfying $\alpha \circ \mathcal{R}=\mathcal{R'} \circ \alpha.$
  \end{defi}
  \begin{thm}
    Let $(L,[\c,\c,\c],\mathcal{R} )$  be $\lambda $-weighted Reynolds  L.t.s. Define $[\cdot,\cdot,\cdot]_ \mathcal{R}$ on $L$ by
    \begin{align}\label{RynoldsBracket}
        [x,y,z]_ \mathcal{R}=  [\mathcal{R}x,\mathcal{R}y,z] + [\mathcal{R}x,y,\mathcal{R}z] + [x,\mathcal{R}y,\mathcal{R}z] +\lambda  [\mathcal{R}x,\mathcal{R}y,\mathcal{R}z]  \quad \forall x,y,z\in L.
    \end{align}
    Then  $(L,[\c,\c,\c]_\mathcal{R},\mathcal{R} )$ is a $\lambda $-weighted Reynolds L.t.s. Moreover $\mathcal{R}$ is a morphism of $\lambda $-weighted Reynolds operators L.t.s from $(L,[\c,\c,\c]_\mathcal{R},\mathcal{R} )$ to $(L,[\c,\c,\c],\mathcal{R} )$, that is $\mathcal{R}([x,y,z]_\mathcal{R})=[\mathcal{R}(x),\mathcal{R}(y),\mathcal{R}(z)].$
  \end{thm}

\begin{pro}\label{PropReynolds}
Let $\mathcal{R}  : L\to L $ be a $\lambda$-weighted Reynolds operator on a Lie algebra, then $\mathcal{R}$ is a $2\lambda$-weighted Reynolds operator on the induced L.t.s.
\end{pro}
 \begin{ex}\label{Expl}
 Let $W$ be the Witt algebra generated by basis elements $\{l_n\}_{n\in \mathbb{Z}}$ and the Lie bracket given by:
\begin{align}
    &[l_m,l_n]= (m-n)l_{m+n} \ \ \ \ \ \forall m,n \in \mathbb{Z}.
\end{align}
By Theorem \ref{Lie.L.t.s}, we obtain a  L.t.s defined by the bracket \begin{align*}
   & [l_m,l_n,l_p]= (m-n)(m+n-p))l_{m+n+p} \ \ \ \ \ \forall m,n,p \in \mathbb{Z}.
\end{align*}
Note that $W_{\geq0} = span\{l_n  | n \geq 0\}$  is a L.t.s subalgebra of $W$.   The linear map $ \mathcal{R} :W_{\geq0} \to W_{\geq0} $ defined by \begin{align}
    \mathcal{R}(l_m)= \frac{- 1}{\lambda(m+1)}\  l_m \ \ \ \ \forall m \in W_{\geq0} \ and \ \lambda\neq 0
\end{align}
is a
$\lambda $-weighted Reynolds operator on the Lie algebra $W_{\geq0}$. Then, according to Proposition \ref{PropReynolds} $\mathcal{R}$ is a  Reynolds operator of weight $2\lambda$ on the induced L.t.s.
This example will be more clear in Section \ref{Section-6} when we will introduce NS-Lie triple systems and a functor
from the category of NS-Lie triple systems to the category of generalized Reynolds operator on L.t.s.\end{ex}


\subsection{Generalized Reynolds operators  on \textsf{L.t.sRep}  pairs}

Let $(L, [\cdot,\cdot, \cdot],\theta)$  be a \textsf{L.t.sRep} pair. For any $3$-cocycle $\mathcal{H}\in \mathcal{Z}_{L.t.s} ^{3}(L , M )$,  there is a Lie triple system structure on the direct sum  of vector spaces $L \oplus M $, defined by:~\small{
\begin{align}\label{twistedsemidirect product}
[x+u,y+v,z+w]_\mathcal{H}=\Big ([x,y,z], \ \theta(y,z)u-\theta(x,z)v+D(x,y)w+ \mathcal{H}(x,y,z)\Big)\quad \forall ~~x,y,z \in L.~~ u,v,w \in M.
\end{align}}
 This  Lie triple system is called the \textbf{twisted semi-direct product  Lie triple system} and denoted by \textbf{$L\ltimes_\theta^{\mathcal{H}} M$}.

In the sequel, $\mathcal{H}$ will always denote a $3$-cocycle.

\begin{defi}
 A linear map  $T  : M \to L $ is said to be a \textbf{ generalized Reynolds operator } on a \textsf{L.t.sRep} pair $(L, [\c,\c,\c],\theta)$  if it satisfies:
\begin{align}\label{O op on lts}
[T u,T v,T w]=T \Big(D(T u,T v)w+\theta(T v,T w)u-\theta(T u,T w)v+ \mathcal{H}(T u,T v,T w)\Big)  \quad \forall u,v,w \in M.
\end{align}
\end{defi}
\begin{ex}Recall that a
\textbf{Rota-Baxter} operator of weight zero  on a
   L.t.s $L$ is a linear map $\mathcal{R} : L \rightarrow L$    satisfying :
\begin{equation*}
    [\mathcal{R}(x),\mathcal{R}(y),\mathcal{R}(z)]=\mathcal{R}\Big([\mathcal{R}(x),\mathcal{R}(y),z]+[\mathcal{R}(x),y,\mathcal{R}(z)]+[x,\mathcal{R}(y),\mathcal{R}(z)]\Big)  \quad \forall x,y,z \in L.
\end{equation*}
Any Rota–Baxter operator of weight zero on a  L.t.s is a
generalized Reynolds operator with respect to the regular representation and  $\mathcal{H}= 0$.
\end{ex}
\begin{rmk} The notion of generalized Reynolds operator   is also called twisted $\mathcal{O}$-operator  or twisted Kupershmidt operator. In \cite{O-operator}, the authors
introduced the notion of an $\mathcal{O}$-operator on a \textsf{L.t.sRep} pair.
It is straightforward to see that generalized Reynolds operators on \textsf{L.t.sRep} pairs are generalization of Rota-Baxter operators of weight 0 and
$\mathcal{O}$-operators on \textsf{L.t.sRep} pairs.
\end{rmk}

\begin{ex}\label{Rey-Twisted}
Any  $\lambda$-weighted Reynolds operator   on a L.t.s. $(L,[\c,\c,\c] )$  is  generalized Reynolds operator  with respect to the regular representation, where $\mathcal{H}$ is given by:
  \begin{align*}
      &\mathcal{H}(x,y,z)  =\lambda [x,y,z] \ \ \ \forall x,y,z \in L.
  \end{align*}
 \end{ex}
\begin{ex}
Let
$(L, [\cdot,\cdot, \cdot],\theta)$ be a \textsf{L.t.sRep} pair. Suppose that  $\varphi \in\mathcal{C}_{L.t.s}^{1}(L, M )$
is an invertible $1$-cochain. Then $T = \varphi^{-1}  : M \to L$ is a generalized Reynolds operator  with $\mathcal{H}= - \delta^{1}\varphi$. Indeed,  we have
\begin{align*}
    &  \delta^{1}\varphi  (Tu,Tv,Tw)= D(Tu,Tv)w-\theta(Tu,Tw)v+\theta(Tv,Tw)u-\varphi ([Tu,Tv,Tw]) \quad  \forall u,v,w \in M.
\end{align*}
By applying $T$ to both sides we get the result.
\end{ex}

\begin{defi}\label{Def}
Let $T$ and $T'$ be two generalized Reynolds operators on a \textsf{L.t.sRep} pair. A morphism of generalized Reynolds operators from $T$
to $T'$ consists of a L.t.s morphism  $\phi : L \to L$ and an endomorphism   $\psi : M \to M$ such that  
\begin{align}
    &\label{c1}\phi \circ T = T^{'}\circ \psi,\\
    &\label{c2}\psi(\theta(x,y)u)=\theta(\phi(x),\phi(y))\psi(u),
    \\&\label{c3}\psi( \mathcal{H}(x,y,z) )=  \mathcal{H}(\phi (x),\phi(y),\phi(z))\quad \forall \ x,y,z \in L, \ u \in M.
\end{align}
In particular, if both $\phi$ and $\psi$ are invertible then $(\phi,\psi)$ is called an \textbf{isomorphism}.

\end{defi}

Let $\varphi \in\mathcal{C}_{L.t.s}^{1}(L, M )$
be a $1$-cochain. Define $\kappa_\varphi  : L\oplus M \to  L\oplus M $ by  $
\kappa_\varphi=\begin{pmatrix}
 Id_L & 0 \\
 \varphi & Id_M
 \end{pmatrix}.
$
\begin{pro}\label{prop}
Let
$(L, [\cdot,\cdot, \cdot],\theta)$ be a \textsf{L.t.sRep} pair and $\varphi \in\mathcal{C}_{L.t.s}^{1}(L, M )$ be a $1$-cochain. Then 	 $\kappa_\varphi$ defined above is an isomorphism from the $\mathcal{H}$-twisted semi-direct product L.t.s $L\ltimes_\theta^{\mathcal{H}} M$ to the ${\mathcal{H}'}$-semi-direct product L.t.s $L\ltimes_\theta^{\mathcal{H}'} M$, where  ${\mathcal{H}'}=\mathcal{H}- \delta^{1} \varphi. $
\end{pro}
The identity \eqref{O op on lts} can be characterized by the graph of $T$ being a subalgebra.
\begin{pro}\label{graph}
A linear map $T  :M\rightarrow L$ is a generalized Reynolds operator on a \textsf{L.t.sRep} pair if and only if  its graph $Gr(T )=\{(T u,u)|\;u\in M\}$ is a subalgebra of the $\mathcal{H}$-twisted semi-direct product  $L\ltimes_\theta^{\mathcal{H}} M$.
\end{pro}
\begin{proof}
Let $T  :M\rightarrow L$ be  a linear map. For all  $u,v w \in M$,  we have:\small{
\begin{align*}
&[T u+u,T v+v,T w+w]_{\mathcal{H}}\\=&
\Big([T u,T v,T w],~~  \theta(T v,T w)u-\theta(T u,T w)v+D(T u,T v)w+ \mathcal{H}(T u,T v,T w)\Big).
\end{align*}}
 which implies that the graph  $Gr(T )$ is a subalgebra of the $\mathcal{H}$-twisted semi-direct product  $L\ltimes_\theta^{\mathcal{H}} M$ if and only  if  $T$ satisfies  \small{
\begin{align*}
&[Tu,Tv,Tw]=
T\Big(   \theta(Tv,Tw)u-\theta(Tu,Tw)v+D(Tu,Tv)w + \mathcal{H}(T u,T v,T w) \Big),
\end{align*}}
which means that $T$ is a generalized Reynolds operator.
\end{proof}
Since $M$ and $ Gr(T)$ are isomorphic as vector spaces, we get the following corollary.
\begin{cor}\label{coro}
Let $T  :M\rightarrow L$  be a generalized Reynolds operator on  a \textsf{L.t.sRep} pair $(L, [\cdot,\cdot, \cdot],\theta)$. Then there is a  \textsf{L.t.s} structure  on $M$ given by:
\begin{align}\label{ltsmod}
    & [u, v,w]_T=\theta(Tv,Tw)u-\theta(Tu,Tw)v+D(Tu,Tv)w+ \mathcal{H}(T u,T v,T w)\quad\forall ~u,v,w \in M.
\end{align}
 Furthermore, $T$ is a  morphism from    $(M,[\cdot,\cdot, \cdot]_T )$ to  $(L,[\cdot,\cdot, \cdot] )$.
\end{cor}

\begin{defi}
Let $T  : M \to L $  be a generalized Reynolds operator  on a  \textsf{L.t.sRep} pair   $(L, [\cdot,\cdot, \cdot],\theta)$.
 The map $\varphi \in  \mathcal{Z}_{L.t.s}^{1}(L , M )$ is called a \textbf{$T$-admissible $1$-cocycle}  if the linear map  $(Id + \varphi \circ  T)  : M\to M$ is invertible.
\end{defi}
\begin{pro}
Let $\varphi$ be a $T$-admissible $1$-cocycle, then
$T \circ (Id + \varphi \circ  T)^{-1}  : M\to L$ is a generalized Reynolds operator, which we denote by  $T_\varphi$.
\end{pro}
\begin{proof}
Let $\varphi$ be a $T$-admissible $1$-cocycle.  By   Propositions \ref{prop} and  \ref{graph},  we have	\begin{equation*}   \kappa_\varphi(Gr(T)) = \{(Tu, u + \varphi (Tu))|u \in M\} \subseteq L\ltimes_\theta^{\mathcal{H}'} M,\end{equation*} is a subalgebra of the $\mathcal{H}'$-twisted semi-direct
product L.t.s  $ L\ltimes_\theta^{\mathcal{H}'} M$, where ${\mathcal{H}'}=\mathcal{H}- \delta^{1} \varphi$. Since the linear map $(Id+ \varphi \circ T) : M \to  M$ is invertible, so	$\kappa_\varphi(Gr(T))$ is
the graph of $T\circ (Id+ \varphi \circ T)^{-1} : M \to L$, which implies that $T\circ (Id+ \varphi \circ T)^{-1} $ is a generalized Reynolds
operator. This completes the proof.
\end{proof}
Recall from Corollary \ref{coro} that a generalized Reynolds operator induces a Lie triple system on $M$.
We have the following proposition.

\begin{pro}
Let $T$ be a generalized Reynolds operator and $\varphi$ be a $T$-admissible $1$-cocycle.
Then the Lie triple system structures  on $M$ induced by $T$ and $T_\varphi$ are isomorphic.
\end{pro}
\begin{proof}
Consider the linear map $(Id  + \varphi \circ T)   : M \to M$. For all $u, v, w \in M$, we have\small {
\begin{eqnarray*}
    &&[(Id +  \varphi \circ T)u,(Id  +  \varphi \circ T)v ,(Id  +  \varphi \circ T) w]_{T_ \varphi} \\
  & = &\theta(T u,T v) (Id  +  \varphi \circ T)w -  \theta(T u,T w)(Id  +  \varphi \circ T) v\\& &+  D(T v,T w) (Id  +  \varphi \circ T)u + \mathcal{H}(T u,T v,T w)\\
  & = &\theta(T u,T v)w   -  \theta(T u,T w) v +  D(T v,T w) u + \mathcal{H}(T u,T v,T w)\\
 & & +\theta(T u,T v)  ( \varphi \circ T)w-  \theta(T u,T w)( \varphi\circ T) v +  D(T v,T w) (  \varphi \circ T)u \\
 & \overset{\eqref{1cocycle}}{=}&[ u,   v,  w]_{T} +  \varphi ([T u,T v,T w])\\
  &=&[ u,   v,  w]_{T} +  \varphi \circ T([ u, v, w]_{T})\\&=&(Id +  \varphi \circ T)[ u,   v,  w]_{T} .
\end{eqnarray*}}
Thus $(Id + \varphi \circ T )$ is an isomorphism of Lie triple systems.
\end{proof}
In Ref.  \cite{Das-Twisted},  the author  introduced a generalized Reynolds operator on LieRep pair $(L,[\c,\c] ,\rho)$ with a $2$-cocycle  $H$ in the Chevalley-Eilenberg  cohomology, as a linear map $T  : M \to L  $  satisfying
 \begin{align*}
     [T u,T v]=T\Big(\rho (T u)v+\rho (T v)u+ H(T u,T v)\Big) \quad \forall u,v \in M.
 \end{align*}

 \begin{pro}
 Let $T : M \to L$ be a generalized Reynolds operator on a LieRep pair $(L,[\c,\c] ,\rho)$. Then $T$ is also a generalized Reynolds operator on the induced \textsf{L.t.sRep} pair.
 \end{pro}
 \section{Yamaguti cohomology of generalized Reynolds operators}\label{Section-4}
 In this section, after recalling the notion of $L_\infty$-algebra, we construct one on a given   graded vector space  whose Maurer-Cartan elements are generalized Reynolds operators on \textsf{L.t.sRep}  pairs. It is the Maurer-Cartan   characterization of generalized Reynolds operator $T$. This characterization allows us to introduce the Yamaguti
cohomology of $T$. Next, we show that the cohomology of $T$ is equivalently described by the Yamaguti cohomology of $M$ with coefficients in a suitable representation   $L$.
 \subsection{$L_\infty$-algebra and Maurer-Cartan characterization}
A permutation $\sigma \in \mathbb S_n $ is called an $(i, n-i)$-shuffle if $\sigma(1) <....<\sigma(i) $ and $\sigma(i+1) <....<\sigma(n) $. If $i = 0$ or $i = n$, we assume $\sigma = Id$. The set of all $(i, n - i)$-shuffles will be
denoted by $\mathbb S_{(i,n-i)}$.
\begin{defi}
An  $L_\infty$-algebra is a $\mathbb Z$-graded vector space $\mathfrak (\mathfrak{g}=\oplus_{k\in\mathbb Z}\mathfrak{g}^k)$ equipped with a collection $(k \ge 1)$ of linear maps $l_k :\otimes^k\mathfrak g\to \mathfrak g$ of degree $1$ with the property that, for any homogeneous elements $x_1,\cdots,x_n\in \mathfrak g$, we have
\begin{itemize}\item[\rm(i)]
graded symmetry:
\begin{eqnarray*}
l_n(x_{\sigma(1)},\cdots,x_{\sigma(n-1)},x_{\sigma(n)})=\varepsilon(\sigma)l_n(x_1,\cdots,x_{n-1},x_n) \quad \forall  \sigma\in\mathbb S_{n},
\end{eqnarray*}
\item[\rm(ii)] generalized Jacobi identity:
\begin{eqnarray*}\label{sh-Lie}
\sum_{i=1}^{n}\sum_{\sigma\in \mathbb S_{(i,n-i)} }\varepsilon(\sigma)l_{n-i+1}(l_i(x_{\sigma(1)},\cdots,x_{\sigma(i)}),x_{\sigma(i+1)},\cdots,x_{\sigma(n)})=0 \quad \forall n \ge 1.
\end{eqnarray*}
\end{itemize}
\end{defi}
\begin{defi}
A Maurer-Cartan element  of an $L_\infty$-algebra $(\mathfrak g=\oplus_{k\in\mathbb Z}\mathfrak g^k,\{l_i\}_{i=1}^{+\infty})$ is an element $\pi\in \mathfrak g^0$ satisfying the Maurer-Cartan equation
\begin{eqnarray}\label{MC-equationL}
\sum_{n=1}^{+\infty} \frac{1}{n!}l_n(\pi,\cdots,\pi)=0.
\end{eqnarray}
\end{defi}
Let $\pi$ be a Maurer-Cartan element of an $L_\infty$-algebra $(\mathfrak{g},\{l_i\}_{i=1}^{+\infty})$. For all $k\geq1$ and $x_1,\cdots,x_n\in \mathfrak g,$
define a series of linear maps $l_k^\pi:\otimes^k \mathfrak g\to \mathfrak g$ of degree $1$ by
\begin{eqnarray}
 l^{\pi}_{k}(x_1,\cdots,x_k)=\sum^{+\infty}_{n=0}\frac{1}{n!}l_{n+k}\{\underbrace{\pi,\cdots,\pi}_n,x_1,\cdots,x_k\}.
\end{eqnarray}

\begin{thm}\label{thm:twist}\cite{Getzler}
With the above notations, $(\mathfrak{g},\{l^{\pi}_i\}_{i=1}^{+\infty})$ is an $L_{\infty}$-algebra, obtained from the  $L_\infty$-algebra  $(\mathfrak{g},\{l_i\}_{i=1}^{+\infty})$ by twisting with the Maurer-Cartan element $\pi$. Moreover, $\pi+\pi'$ is a Maurer-Cartan element of  $(\mathfrak g,\{l_i\}_{i=1}^{+\infty})$ if and only if $\pi'$ is a Maurer-Cartan element of  the twisted $L_{\infty}$-algebra  $(\mathfrak{g},\{l^{\pi}_i\}_{i=1}^{+\infty})$ .
\end{thm}
Let $L$ be a vector space. Consider the graded vector space $C^\ast(L,L)=\oplus_{n\geq 0}   C^n(L,L),$
where  $   C^n(L,L)$ is the set of linear maps $P\in \text{Hom} (\underbrace{(\wedge^{2} L)\otimes \cdots\otimes (\wedge^{2}L)}_{n\geq0}\otimes L ,L),$
 satisfying
\begin{align}
  P(\mathfrak{X}_1,\mathfrak{X}_2,\cdots,\mathfrak{X}_{n-1},x,x,y)=&0, \label{skewsym}  \\
 \circlearrowleft_{x,y,z}  P(\mathfrak{X}_1, \mathfrak{X}_2, \cdots, \mathfrak{X}_{n-1}, x, y, z)=&0. \label{Jac identity}
\end{align}
 The degree of an element in $C^n(L,L)$ is  defined to be $n$. Define
\begin{eqnarray}\label{L.t.s-bracket}
[P,Q]_{L.t.s}=P{\circ}Q-(-1)^{pq}Q{\circ}P \quad \forall~ P\in \mathfrak C^{p}(L,L),Q\in \mathfrak C^{q}(L,L),
\end{eqnarray}
 with $P{\circ}Q\in C^{p+q}(L,L)$ and defined by
{\footnotesize
\begin{equation*}
\begin{aligned}
&(P{\circ}Q)(\mathfrak{X}_1,\cdots,\mathfrak{X}_{p+q},x)\\
=&\sum_{k=1}^{p}(-1)^{(k-1)q}\sum_{\sigma\in \mathcal{S}(k-1,q)}(-1)^\sigma P(\mathfrak{X}_{\sigma(1)},\cdots,\mathfrak{X}_{\sigma(k-1)},
Q(\mathfrak{X}_{\sigma(k)},\cdots,\mathfrak{X}_{\sigma(k+q-1)},x_{k+q})\wedge y_{k+q},\mathfrak{X}_{k+q+1},\cdots,\mathfrak{X}_{p+q},x)\\
&+\sum_{k=1}^{p}(-1)^{(k-1)q}\sum_{\sigma\in \mathcal{S}(k-1,q)}(-1)^\sigma P(\mathfrak{X}_{\sigma(1)},\cdots,\mathfrak{X}_{\sigma(k-1)},x_{k+q}\wedge
Q(\mathfrak{X}_{\sigma(k)},\cdots,\mathfrak{X}_{\sigma(k+q-1)},y_{k+q}),\mathfrak{X}_{k+q+1},\cdots,\mathfrak{X}_{p+q},x)\\
&+\sum_{\sigma\in \mathcal S(p,q)}(-1)^{pq}(-1)^\sigma P(\mathfrak{X}_{\sigma(1)},\cdots,\mathfrak{X}_{\sigma(p)},
Q(\mathfrak{X}_{\sigma(p+1)},\cdots,\mathfrak{X}_{\sigma(p+q-1)},\mathfrak{X}_{\sigma(p+q)},x)),\\
\end{aligned}
\end{equation*}
}
where $\sigma$ is a permutation in $(k-1,q)$-shuffle.
\begin{pro}\cite{O-operator}
The graded vector space $  C^*(L,L)$ equipped with the graded commutator bracket defined by \eqref{L.t.s-bracket} is a graded Lie algebra.
\end{pro}
\begin{pro}
Let $ L$ be a vector space. Then $\pi \in  C^{1}(L,L)$ defines a Lie triple system structure on $L$ if and only if $[\pi ,\pi]_{L.t.s}=0$, i.e., $\pi$ is a Maurer-Cartan element of the graded Lie algebra $ (  C^*(L,L),[\cdot,\cdot]_{L.t.s})$.
Moreover, $(  C^*(L,L),[\cdot,\cdot]_{L.t.s},{\rm d}_{\pi})$ is a differential graded Lie algebra, where ${\rm d}_{\pi}$ is defined by
\begin{eqnarray}
{\rm d}_\pi :=(-1)^{n-1}[\pi,\cdot]_{L.t.s}.
\end{eqnarray}
\end{pro}
Let
$(L, [\cdot,\cdot, \cdot],\theta)$ be a \textsf{L.t.sRep} pair.  For convenience, we use $\pi:\wedge^2L \otimes L  \to L$ to indicate the Lie triple system structure $[\cdot,\cdot,\cdot]$.
Then $\pi+\theta+\mathcal{H}$ corresponds to the semidirect product Lie triple system   structure on $L \oplus M$ given by
\begin{align} [x+u,y+v,z+w]_\mathcal{H}=\Big ([x,y,z], \ \theta(y,z)u-\theta(x,z)v+D(x,y)w+ \mathcal{H}(x,y,z)\Big).
\end{align}
Therefore, we have
$$[\pi+\theta+\mathcal{H},\ \pi+\theta+\mathcal{H}]_{L.t.s}=0.$$
Consider the graded vector space  $  C^\ast(M,L)=\oplus_{n\geq 0} C^n(M,L)$, where $C^n(M,L)$ is the set of linear maps $f \in Hom (\underbrace{(\wedge^{2} M)\otimes \cdots\otimes (\wedge^{2}M)}_{n\geq0}\otimes M, L),$
satisfying
\begin{align}
 f(\mathfrak{U}_1,\mathfrak{U}_2,\cdots,\mathfrak{U}_{n-1},u,u,v)=&0, \\
 \circlearrowleft_{u,v,w} f(\mathfrak{U}_1, \mathfrak{U}_2, \cdots, \mathfrak{U}_{n-1}, u, v, w) =&0, \  \ \forall  \mathfrak{U}_i\in \otimes^{2}M,\;1\leq i\leq n-1.
 \end{align}

Define
\begin{align*}
 & l_3: \mathfrak C^p(M,L) \times \mathfrak C^q(M,L)  \times \mathfrak C^r(M,L) \to \mathfrak C^{p+q+r+1}(M,L), \\
  & l_4: \mathfrak C^p(M,L) \times \mathfrak C^q(M,L)  \times \mathfrak C^r(M,L)\times \mathfrak C^s(M,L) \to \mathfrak C^{p+q+r+s+1}(M,L),
\end{align*}
by
\begin{align*}
 l_3(P,Q,R)=&[[[\pi+\theta,P]_{L.t.s},Q]_{L.t.s},R ]_{L.t.s},\\
l_4(P,Q,R,S)=&[[[[\mathcal{H},P]_{L.t.s},Q]_{L.t.s},R ]_{L.t.s},S]_{L.t.s}.
\end{align*}

Now we  give the Maurer-Cartan characterization of a generalized Reynolds operator on a \textsf{LieRep} pair $(L, [\c,\c,\c],\theta)$.
\begin{pro}
    The graded vector space $C^\ast(M,L)$ is an $L_\infty$-algebra  with
$ \{
     l_1=l_2=0, l_3(\cdot,\cdot,\cdot),  l_4(\cdot,\cdot,\cdot,\cdot)
 \}$
and higher brackets are trivial.
\end{pro}\begin{thm} \label{caracterisation}
   A linear map  $T : M \to L$ is a generalized Reynolds operator  if and only if $T$ is a solution of the Maurer-Cartan equation  of the  $L_\infty$-algebra $(C^\ast(M,L),l_3,l_4)$, i.e.
 $$ \frac{1}{3!}l_3(T,T,T)+\frac{1}{4!}l_4(T,T,T,T)=0.$$
\end{thm}
\begin{proof}
 Using the above discussion, the first part follows. For the second part, we have that for any $T\in Hom(M,L)$,
    \begin{equation}\label{eqproof1}
        l_4(T,T,T,T)(u,v,w)=-24T(\mathcal{H}(Tu,Tv,Tw)).
    \end{equation}
 Next, according to the Proof of Theorem 4.4 in  \cite{O-operator} we have
 \begin{equation}\label{eqproof2}
     l_3(T,T,T)(u,v,w)=6\Big([Tu,Tv,Tw] -T\big(\theta(Tv,Tw)u-\theta(Tu,Tw)v+D(Tu,Tv)w\big)\Big).
 \end{equation}
 Hence we obtain
 \begin{align*}
   &\Big(\frac{1}{3!}l_3(T,T,T)+\frac{1}{4!}l_4(T,T,T,T)\Big)(u,v,w)\\
   =& [Tu,Tv,Tw]-T\big(\theta(Tv,Tw)u-\theta(Tu,Tw)v+D(Tu,Tv)w+ \mathcal{H}(Tu,Tv,Tw)).
 \end{align*}
 Thus, a linear map $T\in Hom(M, L)$ is a generalized Reynolds operator on a \textsf{L.t.sRep} pair $(L, [\c,\c,\c],\theta)$ if and only if $T$ is a  Maurer-Cartan
element of the  $L_\infty$-algebra $(C^\ast(M,L),l_3,l_4)$.
\end{proof}
\begin{pro}
    Let $T$ be a generalized Reynolds operator  on a \textsf{L.t.sRep} pair $(L, [\c,\c,\c],\theta)$. Then $C^\ast(M,L)$ carries a twisted $L_\infty$-algebra structure given by
    \begin{align}
        &l_1^T(P)=\frac{1}{2}l_3(T,T,P)+\frac{1}{6}l_4(T,T,T,P),\\
        &l_2^T(P,Q)=l_3(T,P,Q)+\frac{1}{2}l_4(T,T,P,Q),\\
        &l_3^T(P,Q,R)=l_3(P,Q,R)+l_4(T,P,Q,R),\\
        & l_4^T(P,Q,R,S)=l_4(P,Q,R,S),\\
        &l_k^T=0,\quad k \geq 5,
    \end{align}
    where $P \in C^p(M,L), Q \in C^q(M,L), R \in C^r(M,L)$ and  S  $\in C^s(M,L)$.
 Moreover,  for any linear map $T' : M \to L$, the sum $T+T'$ is a generalized Reynolds operator if and only if $T'$ is a Maurer-Cartan element in the twisted $L_\infty$-algebra $(C^\ast(M,L),l_1^T,l_2^T,l_3^T,l_4^T)$, i.e  satisfies
 $$l_1^T(T')+\frac{1}{2!}l_2^T(T',T')+\frac{1}{3!}l_3^T(T',T',T')+\frac{1}{4!}l_4^T(T',T',T',T')=0.$$

\end{pro}
\begin{proof}
  For the first part,   since $T$ is a Maurer-Cartan element of the  $L_\infty$-algebra $(C^\ast(M,L),l_3,l_4)$, by Theorem \ref{thm:twist},
we have that  $C^\ast(M,L)$ carries a twisted $L_\infty$-algebra structure.  For the second part, by Theorem \ref{caracterisation}, $T+T'$ is a generalized Reynolds operator if and only if
\begin{equation}
    \frac{1}{3!}l_3(T+T',T+T',T+T')+\frac{1}{4!}l_4(T+T',T+T',T+T',T+T')=0.
\end{equation}
Applying $\frac{1}{3!}l_3(T,T,T)+\frac{1}{4!}l_4(T,T,T,T)=0$, the above condition is equivalent to
\begin{align*}
    &\frac{1}{3!}\Big(3l_3(T,T,T')+3l_3(T,T',T')+l_3(T',T',T')\Big)\\
    &+\frac{1}{4!}\Big(4l_4(T,T,T,T')+6l_4(T,T,T',T')+4l_4(T,T',T',T')+l_4(T',T',T',T')\Big)=0,
\end{align*}
that is,  $l_1^T(T')+\frac{1}{2!}l_2^T(T',T')+\frac{1}{3!}l_3^T(T',T',T')+\frac{1}{4!}l_4^T(T',T',T',T')=0$, which implies that $T'$ is a Maurer-Cartan element of the twisted $L_\infty$-algebra $(C^\ast(M,L),l_1^T,l_2^T,l_3^T,l_4^T)$.
\end{proof}
The above characterization of a generalized Reynolds operator $T$ allows us to define a cohomology
associated to $T$. More precisely, we define
$C^n_T(M,L)=Hom (\underbrace{\otimes^{2} M\otimes \cdots\otimes (\otimes^{2}M)}_{n\geq0}\wedge M,L)$, for $n \geq 0$ and the differential operator $d_T : C^n_T(M,L) \to C^{n+1}_T(M,L)$ by
\begin{align}\label{coboumaurer}
    d_T(f)=l_1^T(f)=\frac{1}{2}l_3(T,T,f)+\frac{1}{6}l_4(T,T,T,f).
\end{align}
The corresponding cohomology groups  are
\begin{equation*}
  H^n_T(M,L)=\frac{Z^n_T(M,L)}{B^n_T(M,L)}=\frac{\{f \in C^n_T(M,L)|d_T(f)=0\}}{\{d_T(g)| g \in C^{n-1}_T(M,L)\}}.
\end{equation*}
\subsection{Yamaguti cohomology }

Let $T :M \to L $ be a generalized Reynolds operator on a \textsf{L.t.sRep} pair $ (L, [\cdot,\cdot,\cdot],\theta)$.  Once  a \textsf{L.t.s}  structure on the representation space $M$ is given,  we construct a representation of the representation space (viewed as a L.t.s ) on the \textsf{L.t.sRep} pair (viewed as a vector space) as follow:
Define
the linear map $\theta_T : \otimes^2 M \to End(L)$ by
\begin{equation}\label{inducedRep}
   \theta_T(u,v)x = [x,Tu,Tv]- T \Big( D(x,Tu)v- \theta(x,Tv)u +\mathcal{H} (x,Tu,Tv) \Big) \quad \forall u,v\in M ~  and ~ x\in L.
\end{equation}

\begin{pro} With the above  notations, $(L,\theta_T)$ is a representation  of the  L.t.s  $(M,[\cdot, \cdot, \cdot]_T).$
\end{pro}\begin{proof}
Note that \begin{align*}
    D_T(u,v)x=&  \theta_T(v,u)x- \theta_T(u,v)x\\
    =&[Tu,Tv,x]- T(\theta(Tv,x))u -\theta(Tu,x))v+ \mathcal{H}(Tu,Tv,x))   \quad\forall u,v\in M ~  and ~ x\in L.
\end{align*}
According to
Eqs. \eqref{lts 1}-\eqref{rep lts 3}, \eqref{2.11}- \eqref{2.12} and   \eqref{O op on lts}-\eqref{ltsmod}, for all $u_i \in M, 1 \leq i \leq 4$ and   $x\in L$, we get
\begin{align*}
    &\theta_T(u_3,u_4) \theta_T(u_1,u_2) - \theta_T(u_2,u_4) \theta_T(u_1,u_3)-\theta_T(u_1,[u_2,u_3,u_4]_T) +D_T(u_2,u_3)\theta_T(u_1,u_4)   x\\
   =& [[x,Tu_1,Tu_2], Tu_3,Tu_4] + [Tu_2,[x,Tu_1,Tu_3], Tu_4]+ [Tu_2, Tu_3,[x,Tu_1,Tu_4]]  \\ & - [x,Tu_1,[Tu_2, Tu_3,Tu_4]] + T \Big (\theta(Tu_3,Tu_4)\theta(x,Tu_2)u_1-\theta(Tu_3,Tu_4)\mathcal{H}(x,Tu_1,Tu_2) \\
   &-\theta(Tu_3,Tu_4)D(x,Tu_1)u_2 + \theta(Tu_2,Tu_4)\mathcal{H}(x,Tu_1,Tu_3)  - \theta(Tu_2,Tu_4)\theta(x,Tu_3)u_1\\
   & + \theta(Tu_2,Tu_4)D(x,Tu_1)u_3 + D(x,Tu_1)\theta(Tu_3,Tu_4)u_2-D(x,Tu_1) \theta(Tu_2,Tu_4)u_3\\&+D(x,Tu_1)D(Tu_2,Tu_3)u_4+ D(x,Tu_1)\mathcal{H}(T u_2,T u_3,T u_4)    +\mathcal{H}(x,Tu_1,[Tu_2,Tu_3,Tu_4])\\
  & - \theta(x,[Tu_2,Tu_3,Tu_4])u_1 + D(Tu_3,Tu_2) \mathcal{H}(x,Tu_1,Tu_4) + D(Tu_2,Tu_3)\theta(x,Tu_4)u_1\\
  & + D(Tu_3,Tu_2)D(x,Tu_1)u_2- D([x,Tu_1,Tu_2],Tu_3)u_4+ D([x,Tu_1,Tu_2],Tu_4)u_3\\
  & +\theta([x,Tu_1,Tu_2],Tu_4)u_3-\theta([x,Tu_1,Tu_3],Tu_4)u_2  - D([x,Tu_1,Tu_4],Tu_3)u_2\\
  &  +\theta([x,Tu_1,Tu_4],Tu_2)u_3+ D([x,Tu_1,Tu_4],Tu_2)u_3- \theta([x,Tu_1,Tu_4],Tu_3)u_2\Big  )\\
  =& T \Big ( \mathcal{H}(x,Tu_1,[Tu_2,Tu_3,Tu_4])\\  &+D(x,Tu_1)\mathcal{H}(Tu_2,Tu_3,Tu_4  )- \mathcal{H}([x,Tu_1,Tu_2],Tu_3,Tu_4) \\
   & - \mathcal{H}(Tu_2,[x,Tu_1,Tu_3],Tu_4) - \mathcal{H}(Tu_2,Tu_3,[x,Tu_1,Tu_4]) - \theta(Tu_3,Tu_4) \mathcal{H}(x,Tu_1,Tu_4)\\& +\theta(Tu_2,Tu_4) \mathcal{H}(x,Tu_1,Tu_3)-
   D(Tu_2,Tu_3) \mathcal{H}(x,Tu_1,Tu_4) \Big)\\
  =& T (\delta^{3}\mathcal{H}(x,Tu_1,Tu_2,Tu_3,Tu_4)) =0.
\end{align*}
Similarly, we have
\begin{align*}
 &\theta_T(u_3,u_4) D_T(u_1,u_2) - D_T(u_1,u_2) \theta_T(u_3,u_4)+\theta_T([u_1,u_2,u_3]_T,u_4) +\theta_T(u_3,[u_1,u_2,u_4]_T )  x\\
  =& [[Tu_1,Tu_2,x], Tu_3,Tu_4]+[x,[Tu_1,Tu_2, Tu_3],Tu_4] + [x,Tu_3,[Tu_1,Tu_2, Tu_4]]\\ &- [Tu_1, Tu_2,[x,Tu_3,Tu_4]]    + T \Big (\theta(Tu_3,Tu_4)\theta(Tu_1,x)u_2- \theta(Tu_3,Tu_4)\theta(Tu_2,x)u_1 \\ &- \theta(Tu_3,Tu_4)\mathcal{H}(Tu_1,Tu_2,x)- \mathcal{H} ([Tu_1,Tu_2,x],Tu_3,Tu_4 ) + D (Tu_1,Tu_2)\mathcal{H}(x,Tu_3,Tu_4)\\ &-  D (Tu_1,Tu_2)\theta (x,Tu_4)u_3 +D (Tu_1,Tu_2)D (x,Tu_3)u_4 +\mathcal{H} (Tu_1,Tu_2,[x,Tu_3,Tu_4 ])\\ &+ \theta(x,Tu_4)\theta(Tu_2,Tu_3)u_1  -   \theta(x,Tu_4)\theta(Tu_1,Tu_3)u_2   + \theta(x,Tu_4)D(Tu_1,Tu_2)u_3\\ & +   \theta(x,Tu_4)\mathcal{H}(Tu_1,Tu_2,Tu_3) - D(x,[Tu_1,Tu_2,Tu_3])u_4   - \mathcal{H}(x,[Tu_1,Tu_2,Tu_3],Tu_4)\\ &
+ \theta(x,[Tu_1,Tu_2,Tu_4])u_3  - D(x,Tu_3) \theta(Tu_2,Tu_4)u_1 + D(x,Tu_3)  \theta(Tu_1,Tu_4)u_2\\ & - D(x,Tu_3)  D(Tu_1,Tu_2)u_4- D(x,Tu_3)  \mathcal{H}(Tu_1,Tu_2,Tu_1)
 - \mathcal{H}(x,Tu_3,[Tu_1,Tu_2,Tu_4]) \\ &- D([Tu_1,Tu_2,x],Tu_3)u_4 +
\theta([Tu_1,Tu_2,x],Tu_4)u_3  - \theta(Tu_2,[x,Tu_3,Tu_4])u_1 \\& - \theta(Tu_1,[x,Tu_3,Tu_4])u_2 \Big )\\
  =& T \Big ( \mathcal{H} (Tu_1,Tu_2,[x,Tu_3,Tu_4 ])- \mathcal{H} ([Tu_1,Tu_2,x],Tu_3,Tu_4 )+D(Tu_1,Tu_2)\mathcal{H} (x,Tu_3,Tu_4) \\ & -  \mathcal{H} (x,[Tu_1,Tu_2,Tu_3],Tu_4 ) -  \mathcal{H} (x,Tu_3,[Tu_1,Tu_2,Tu_4] )- \theta (Tu_3,Tu_4) \mathcal{H} (Tu_1,Tu_2,x)\\ &  + \theta(x,Tu_4) \mathcal{H}(Tu_1,Tu_2,Tu_3)- D(x,Tu_3) \mathcal{H}(Tu_1,Tu_2,Tu_4)\Big )\\
  =& T (\delta^{3}\mathcal{H}(Tu_1,Tu_2,x,Tu_3,Tu_4)) =0.
\end{align*}
Therefore, we deduce that $\theta_T$ is a representation  of the L.t.s $(M,[\cdot, \cdot, \cdot]_T)$ on $L$.
\end{proof}
It follows from the above Proposition that we may consider The Yamaguti cohomology of  L.t.s $(M,[\cdot,\cdot, \cdot]_T )$ with
coefficients in the representation $(L,\theta_T)$. More precisely, for each $n \geqslant 0$, we denote by $\mathfrak{C}^{2n+1}(M,L)$ the Yamaguti   $(2n+1)$-cochains of $M$ with coefficients in $L$, that a $(2n+1)$-cochain $\varphi\in \mathfrak{C}^{2n+1}(M,L)$ is a linear map of $\otimes^{2n+1}  M$ into $L$ satisfying
\begin{align*}
   & \varphi(v_1,v_2,\cdots,v_{2n-2},v,v,u)=0,\\
&\circlearrowleft_{u,v,w}\varphi(v_1, v_2, \cdots, v_{2n-2}, u, v, w)  = 0.\end{align*}
Define the corresponding coboundary operator  $\delta^{2n-1}_T : \mathfrak{C}^{2n-1}(M,L) \to \mathfrak{C}^{2n+1}(M,L)$ by
\begin{align*}
  &  \delta^{2n-1}_T \varphi(v_1,v_2, \cdots , v_{2n+1})\nonumber\\
    =&\theta_T(v_{2n},v_{2n+1})\varphi(v_1,v_2, \cdots , v_{2n-1})- \theta_T(v_{2n-1},v_{2n+1})\varphi(v_1,v_2, \cdots , v_{2n-2},v_{2n}) \nonumber \\
    &+ \sum_{k=1}^n (-1)^{n+k}D_T(v_{2k-1},v_{2k})\varphi(v_1,v_2, \cdots , \widehat{v}_{2k-1},\widehat{v}_{2k}, \cdots , v_{2n+1}) \nonumber \\
    &+ \sum_{k=1}^n \sum_{j=2k+1}^{2n+1}  (-1)^{n+k+1} \varphi(v_1,v_2, \cdots,  \widehat{v}_{2k-1},\widehat{v}_{2k}, \cdots, [v_{2k-1},v_{2k},v_j]_T,\cdots , v_{2n+1})\\
    =& [\varphi(v_1,v_2, \cdots , v_{2n-1}),Tv_{2n},Tv_{2n+1} ] - TD(\varphi(v_1,v_2, \cdots , v_{2n-1}),Tv_{2n}) v_{2n+1}\\
    &+T\theta(\varphi(v_1,v_2, \cdots , v_{2n-1}),Tv_{2n+1}) v_{2n}- T \mathcal{H}(\varphi(v_1,v_2, \cdots , v_{2n-1}),Tv_{2n},Tv_{2n+1})\\
    &-  [\varphi(v_1,v_2, \cdots , v_{2n-2},v_{2n}),Tv_{2n-1},Tv_{2n+1} ]+ TD(\varphi(v_1, \cdots , v_{2n-2},v_{2n}),Tv_{2n-1}) v_{2n+1}\\
    &-T\theta(\varphi(v_1, \cdots , v_{2n-2},v_{2n}),Tv_{2n+1}) v_{2n-1}+T \mathcal{H}(\varphi(v_1, \cdots , v_{2n-2},v_{2n}),Tv_{2n-1},Tv_{2n+1})\\
    &+\sum_{k=1}^n (-1)^{n+k}\Big([T(v_{2k-1},Tv_{2k},\varphi(v_1, \cdots , \widehat{v}_{2k-1},\widehat{v}_{2k}, \cdots , v_{2n+1})]\\&\quad - T\theta(Tv_{2k},\varphi(v_1, \cdots , \widehat{v}_{2k-1},\widehat{v}_{2k}, \cdots , v_{2n+1})v_{2k-1}+ T\theta(Tv_{2k-1},\varphi(v_1, \cdots , \widehat{v}_{2k-1},\widehat{v}_{2k}, \cdots , v_{2n+1}) v_{2k}\\
   &\quad - T \mathcal{H}(Tv_{2k-1},Tv_{2k},\varphi(v_1, \cdots , \widehat{v}_{2k-1},\widehat{v}_{2k}, \cdots , v_{2n+1}))\Big)\\
   &+ \sum_{k=1}^n \sum_{j=2k+1}^{2n+1}  (-1)^{n+k+1}\Big( \varphi(v_1, \cdots,  \widehat{v}_{2k-1},\widehat{v}_{2k}, \cdots, \theta (Tv_{2k},Tv_{j})v_{2k-1}-\theta (Tv_{2k-1},Tv_{j})v_{2k},\cdots , v_{2n+1})\\
   &\quad +\varphi(v_1,v_2, \cdots,  \widehat{v}_{2k-1},\widehat{v}_{2k}, \cdots, D(Tv_{2k-1},Tv_{2k}) v_{j}+ \mathcal{H}(Tv_{2k-1},Tv_{2k},Tv_{j} ),\cdots , v_{2n+1})\Big).
\end{align*}
With this coboundary operator the Yamaguti cochain forms a complex
$$ \mathfrak{C}^{1}(M,L)\overset{\delta_T^{1}}{\longrightarrow} \mathfrak{C}^{3}(M,L)\overset{\delta_T^{3}}{\longrightarrow}  \mathfrak{C}^{5}(M,L) \longrightarrow \cdots,$$
such that $\delta_T^{2n+1}\circ \delta_T^{2n-1}=0$ for all  $n \geq 1 $. In particular,
a $1$-cochain $\varphi \in\mathfrak{C}^{1}(M, L )$
is  $1$-cocycle if

\begin{align}
    &  T \mathcal{H}(\varphi(v_1),Tv_{2},Tv_{3})\nonumber\\
    =&[\varphi(v_1),Tv_2,Tv_3 ]+[Tv_1,\varphi(v_2),Tv_3 ]+[Tv_1,Tv_2,\varphi(v_3) ] + T\theta(\varphi(v_1),Tv_3)v_2\nonumber\\
     &- TD(\varphi(v_1), Tv_2)v_3 + TD(\varphi(v_2), Tv_1)v_3- T\theta(\varphi(v_2),Tv_3)v_1+  T \mathcal{H}(\varphi(v_2),Tv_{1},Tv_{3})\nonumber
     \\
     &- TD(\varphi(v_1), Tv_2)v_3 - T\theta(Tv_2,\varphi(v_3))v_1+ T\theta(Tv_1,\varphi(v_3))v_2-  T \mathcal{H}(Tv_{1},Tv_{2},\varphi(v_3))\nonumber\\
     &-\varphi(\theta(Tv_2,Tv_3)v_1)-\varphi(\theta(Tv_1,Tv_3)v_2)+\varphi(D(Tv_1,Tv_2)v_3)+\varphi(\mathcal{ H}(Tv_1,Tv_2,Tv_3).\label{1.cocycle}
\end{align}

For all $\chi =(x_1,x_2) \in L  \otimes L $, we define $\partial_T(\chi) :M \rightarrow L$ by 
 \begin{equation}
    \partial_T(\chi)u=T\Big(D(\chi)u + \mathcal{H}(\chi,Tu)\Big)-[\chi,Tu]  \quad \forall  u\in M.
 \end{equation}
\begin{pro}
 Let $T$ be a generalized Reynolds operator on a \textsf{L.t.sRep} pair
$(L, [\cdot,\cdot, \cdot],\theta)$. Then $\partial_T(\chi)$ is a $1$-cocycle on the L.t.s $(M,[\c,\c,\c]_T)$ with coefficients in $(L,\theta_T)$.
\end{pro}
\begin{proof}
 For any $u_1,u_2,u_3\in M$, we have:{\small
 \begin{align*}
     &\delta^{1}_T\circ\partial_T(\chi)(u_1,u_2, u_{3}) \\
   =  & [TD(\chi)u_1,Tu_2,Tu_3]+  [Tu_1,TD(\chi)u_2,Tu_3]+  [Tu_1,Tu_2,TD(\chi)u_3]+   [T\mathcal{H}(\chi,Tu_1),Tu_2,Tu_3]\\
     &+[Tu_1,T\mathcal{H}(\chi,Tu_2),Tu_3]+[Tu_1,Tu_2,T\mathcal{H}(\chi,Tu_3)]- [[\chi,Tu_1],Tu_2,Tu_3]
     \\& - [Tu_1,[\chi,Tu_2],Tu_3]-  [Tu_1,Tu_2,[\chi,Tu_3]]+[\chi,[Tu_1,Tu_2,Tu_3]]
     \\&- T\Big( D(TD(\chi)u_1,Tu_2)u_3+D(T\mathcal{H}(\chi)u_1,Tu_2)u_3 - D([\chi,Tu_1],Tu_2)u_3\\
     &-  \theta(TD(\chi)u_1,Tu_3)u_2-\theta (T\mathcal{H}(\chi)u_1,Tu_3)u_2 + \theta([\chi,Tu_1],Tu_3)u_2\\
      &+  \mathcal{H}(TD(\chi)u_1,Tu_2,Tu_3)+\mathcal{H} (T\mathcal{H}(\chi,Tu_1),Tu_2,Tu_3) - \mathcal{H}([\chi,Tu_1],Tu_2,Tu_3)\\
      &-  D(TD(\chi)u_2,Tu_1)u_3-D (T\mathcal{H}(\chi,Tu_2),Tu_1)u_3 +D ([\chi,Tu_2],Tu_1)u_3\\
      &+ \theta(TD(\chi)u_2,Tu_3)u_1+\theta (T\mathcal{H}(\chi,Tu_2),Tu_3)u_1 -\theta ([\chi,Tu_2],Tu_3)u_1\\
       &+\mathcal{H}(Tu_1,TD(\chi)u_2,Tu_3)-\mathcal{H} (Tu_1,T\mathcal{H}(\chi,Tu_2),Tu_1) +\mathcal{H} (Tu_1,[\chi,Tu_2],Tu_3)\\
        &+ \theta(Tu_2,TD(\chi)u_3)u_1+\theta (Tu_2,T\mathcal{H}(\chi,Tu_3))u_1 -\theta (Tu_2,[\chi,Tu_3])u_1\\
         &- \theta(Tu_1,TD(\chi)u_3)u_2-\theta (Tu_1,T\mathcal{H}(\chi,Tu_3))u_2 +\theta (Tu_1,[\chi,Tu_3])u_2\\
         &+\mathcal{H} (Tu_1,Tu_2,TD(\chi)u_3 )+ \mathcal{H}(Tu_1,Tu_2,T\mathcal{H}(\chi,Tu_3))-\mathcal{H} (Tu_1,Tu_2,[\chi,Tu_3])\Big)\\
         &-TD(\chi)(\theta(Tu_2,Tu_3 )u_1 - \theta(Tu_1,Tu_3 )u_2+D(Tu_1,Tu_2 )u_3 + \mathcal{H}(Tu_1,Tu_2,Tu_3))\\
         &-\mathcal{H}(\chi,[Tu_1,Tu_2,Tu_3])+ [\chi,[Tu_1,Tu_2,Tu_3]]\\
          \overset{\eqref{lts 2}\eqref{O op on lts}}{=}& T\Big( \theta(Tu_2,Tu_3 )D(\chi)u_1 - \theta(Tu_1,Tu_3 )D(\chi)u_2+ D(Tu_1,Tu_2 )D(\chi)u_3\\
          & - \theta(Tu_1,Tu_3 )\mathcal{H}(\chi,T u_2)+D(Tu_1,Tu_2 )\mathcal{H}(\chi,\mathcal T u_3)+ D([\chi,T u_1],T u_2)u_3-
          \theta([\chi,T u_1],T u_3)u_2\\
          & + \mathcal{H}([\chi,T u_1],T u_2,T u_3)- D([\chi,T u_2],T u_1)u_3+\theta([\chi,T u_2],T u_2)u_1 + \mathcal{H}(T u_1,[\chi,T u_2],T u_3)\\&+ \theta(T u_2,[\chi,T u_3],)u_1- \theta(T u_1,[\chi,T u_3],)u_2+ \mathcal{H}(T u_1,T u_2,[\chi,T u_3])-
           \mathcal{H}(\chi,[T u_1,T u_2,T u_3])\\
          & -D(\chi)\theta(Tu_2,Tu_3 )u_1 + D(\chi)\theta(Tu_1,Tu_3 )u_2-D(\chi)D(Tu_1,Tu_2 )u_3 -D(\chi) \mathcal{H}(Tu_1,Tu_2,Tu_3)) \Big)\\ \overset{\eqref{2.12}}{=}&
          T\Big( (\theta(Tu_2,Tu_3 )D(\chi) + \theta(T u_2,[\chi,T u_3])+\theta([\chi,T u_2],T u_2)-D(\chi)\theta(Tu_2,Tu_3 ))u_1\\&- (\theta(Tu_1,Tu_3 )D(\chi)-
          \theta([\chi,T u_1],T u_3)- \theta(T u_1,[\chi,T u_3])+ D(\chi)\theta(Tu_1,Tu_3 ))u_2\\&
          +( D(Tu_1,Tu_2 )D(\chi)
          + D([\chi,T u_1],T u_2)
          - D([\chi,T u_2],T u_1)
          -D(\chi)D(Tu_1,Tu_2 ))u_3 \Big)\\
          \overset{\eqref{rep lts 2}\eqref{rep lts 3}}{=}&0.
         \end{align*}}
         Thus, we deduce that $\partial_T(\chi)$ is a $1$-cocycle.
\end{proof}

Define the set of $(2n-1)$-cochains by 
 \begin{equation}
     \mathcal{C}_{T}^{2n-1}(M,L)=\begin{cases}

     L\otimes L \quad \quad \quad\;\; \; \ if \ n=0,\\\mathfrak{C}^{2n-1}(M,L)\quad if \  n\geq 1.
     \end{cases}
 \end{equation}
 Define $\Lambda_ T :  \mathcal{C}_{ T}^{2n-1}(M,L) \to  \mathcal{C}_{ T}^{2n+1}(M,L)$ by 
 \begin{equation}
     \Lambda_T=\begin{cases}

     \partial_ T \;\;\;\ \ if \ n=0,\\\delta_ T \quad \ if \ n\geq 1.
     \end{cases}
 \end{equation}
Now we give the cohomology of generalized Reynolds operator  on a  \textsf{L.t.sRep} pair   $(L, [\cdot,\cdot, \cdot],\theta)$.
\begin{defi}
 Let $T$ be a generalized Reynolds operator on a \textsf{L.t.sRep} pair
$(L, [\cdot,\cdot, \cdot],\theta)$.
 Denote the set of cocycles by ${\mathcal{Z}}^{\ast}(M,L)$, the set of coboundaries by ${\mathcal{B}}^{\ast}(M,L)$
 and the cohomology group by
 $${\mathbf{H}}_T^{\ast}(M,L)={\mathcal{Z}}^{\ast}(M,L) / {\mathcal{B}}^{\ast}(M,L).$$
 The cohomology groups correspond  to  the  cohomology groups for the generalized Reynolds operator $T$.
 \end{defi}
The coboundary operator $\Lambda_T$ coincides with the differential $d_T$ defined by
\eqref{coboumaurer} using the Maurer-Cartan element $T$   of the  $L_\infty$-algebra $(C^\ast(M,L),l_3,l_4)$.
\begin{thm}
 Let $T$ be a generalized Reynolds operator  on a \textsf{L.t.sRep} pair
$(L, [\cdot,\cdot, \cdot],\theta)$. Then we have
\begin{equation}
   d_T(f)= (-1)^{n-1}\Lambda_T(f)\quad \forall f\in Hom (\underbrace{M\otimes \cdots\otimes M}_{2n-1}, L),\;n\geq1.
\end{equation}
\end{thm}
\begin{proof}
    In \cite{{O-operator}}, the authors showed that
   \begin{align*}
&\frac{1}{2}l_3(T,T,f)(\mathfrak{U}_1,\cdots,\mathfrak{U}_{n},u_{n+1})\\
=&[[[\pi+\theta,T]_{L.t.s},T]_{L.t.s},f]_{L.t.s}(\mathfrak{U}_1,\cdots,\mathfrak{U}_{n},u_{n+1})\\
=&(-1)^{n-1}[[\pi+\theta,T]_{L.t.s},T]_{L.t.s}\Big(f(\mathfrak{U}_1,\cdots,\mathfrak{U}_{n-1},u_n),v_n,u_{n+1}\Big)\\
&+(-1)^{n-1}[[\pi+\theta,T]_{L.t.s},T]_{L.t.s}\Big(u_n,f(\mathfrak{U}_1,\cdots,\mathfrak{U}_{n-1},v_n),u_{n+1}\Big)\\
&+(-1)^{n-1}\sum_{i=1}^n(-1)^{n-1}(-1)^{i-1}[[\pi+\theta,T]_{L.t.s},T]_{L.t.s}\Big(\mathfrak{U}_i,f(\mathfrak{U}_1\cdots,\widehat{\mathfrak{U}_i},\cdots,\mathfrak{U}_{n},u_{n+1})\Big)\\
&-\sum_{k=1}^{n-1}\sum_{i=1}^k(-1)^{i+1}f\Big(\mathfrak{U}_1\cdots,\widehat{\mathfrak{U}_i},\cdots,\mathfrak{U}_{k},[[\pi+\theta,T]_{L.t.s},T]_{L.t.s}(\mathfrak{U}_i,u_{k+1}),v_{k+1},\mathfrak{U}_{k+2},\cdots,\mathfrak{U}_n,u_{n+1}\Big)\\
&-\sum_{k=1}^{n-1}\sum_{i=1}^k(-1)^{i+1}f\Big(\mathfrak{U}_1\cdots,\widehat{\mathfrak{U}_i},\cdots,\mathfrak{U}_{k},u_{k+1},[[\pi+\theta,T]_{L.t.s},T]_{L.t.s}(\mathfrak{U}_i,v_{k+1}),\mathfrak{U}_{k+2},\cdots,\mathfrak{U}_n,u_{n+1}\Big)\\
&-\sum_{i=1}^{n}f\Big(\mathfrak{U}_1\cdots,\widehat{\mathfrak{U}_i},\cdots,\mathfrak{U}_{n},[[\pi+\theta,T]_{L.t.s},T]_{L.t.s}(\mathfrak{U}_i,u_{n+1})\Big).
\end{align*}
 Now,  we add one argument and show that 
 \begin{align*}
     &l_4(T,T,T,f)(\mathfrak{U}_1,\cdots,\mathfrak{U}_{n},u_{n+1})\\
     =&[\mathcal{H},T]_{L.t.s},T]_{L.t.s},T]_{L.t.s},f]_{L.t.s}(\mathfrak{U}_1,\cdots,\mathfrak{U}_{n},u_{n+1})\\
     =&[\mathcal{H},T]_{L.t.s},T]_{L.t.s},T]_{L.t.s}(f(\mathfrak{U}_1,\cdots,\mathfrak{U}_{n-1},u_{n}), v_{n},u_{n+1})\\
     &+[\mathcal{H},T]_{L.t.s},T]_{L.t.s},T]_{L.t.s}(u_{n}, f(\mathfrak{U}_1,\cdots,\mathfrak{U}_{n-1},v_{n}),u_{n+1})\\
     &+\sum_{i=1}^{n}(-1)^{n-1}(-1)^{i-1}[\mathcal{H},T]_{L.t.s},T]_{L.t.s},T]_{L.t.s}\Big(\mathfrak{U}_{i},f(\mathfrak{U}_1,\cdots,\widehat{\mathfrak{U}_{i}},\cdots,\mathfrak{U}_{n},u_{n+1})\Big)\\
     &-(-1)^{n-1}\sum_{k=1}^{n-1}\sum_{i=1}^{k}(-1)^{i+1}
     f\Big(\mathfrak{U}_1,\cdots,\widehat{\mathfrak{U_i}},\cdots,\mathfrak{U}_k,[\mathcal{H},T]_{L.t.s},T]_{L.t.s},T]_{L.t.s}(\mathfrak{U}_i,u_{k+1}), v_{k+1},\mathfrak{U}_{k+2},\cdots,\mathfrak{U}_n,u_{n+1}\Big)\\
     &-(-1)^{n-1}\sum_{k=1}^{n-1}\sum_{i=1}^{k}(-1)^{i+1}
     f\Big(\mathfrak{U}_1,\cdots,\widehat{\mathfrak{U_i}},\cdots,\mathfrak{U}_k,u_{k+1},[\mathcal{H},T]_{L.t.s},T]_{L.t.s},T]_{L.t.s}(\mathfrak{U}_i,v_{k+1}),\mathfrak{U}_{k+2},\cdots,\mathfrak{U}_n,u_{n+1}\Big)\\
     &-(-1)^{n-1}\sum_{i=1}^{n}(-1)^{i+1}f\Big(\mathfrak{U}_1,\cdots,\widehat{\mathfrak{U_i}},\cdots,\mathfrak{U}_n,[\mathcal{H},T]_{L.t.s},T]_{L.t.s},T]_{L.t.s}(\mathfrak{U}_i,u_{n+1})\Big)\\
     =&(-1)^{n-1}6\Bigg\{(-1)^{n+1}\Big(
-T\mathcal{H}(f(\mathfrak{U}_1,\cdots,\mathfrak{U}_{n-1},u_{n}),Tv_{n},Tu_{n+1})
-T\mathcal{H}(f(\mathfrak{U}_1,\cdots,\mathfrak{U}_{n-1},v_{n}),Tu_{n+1},Tu_{n})\Big)\nonumber\\
&-\sum_{j=1}^{n}(-1)^{j+1}
T\mathcal{H}(f(\mathfrak{U}_1,\cdots,\widehat{\mathfrak{U}_{j}},\cdots,\mathfrak{U}_{n},u_{n+1}),Tu_{j},T_{j})\nonumber\\
&+\sum_{j=1}^{n}(-1)^{j}f\Big(\mathfrak{U}_1,\cdots,\widehat{\mathfrak{U}_{j}},\cdots,\mathfrak{U}_n,\mathcal{H}(Tu_{j},Tv_{j},Tu_{n+1})\Big)\\
&+\sum_{1\leq j < k \leq n}(-1)^{j}
f\Big(\mathfrak{U}_1,\cdots,\widehat{\mathfrak{U}_j},\cdots,\mathfrak{U}_{k-1},\mathcal{H}(Tu_j,Tv_j,Tu_k), v_k+u_k, \mathcal{H}(Tu_j,Tv_j,Tv_k),\mathfrak{U}_{k+1},\cdots,\mathfrak{U}_n,u_{n+1}\Big)\Bigg\}.
 \end{align*}
 Hence $d_T(f)=\frac{1}{2}l_3(T,T,f)+\frac{1}{6}l_4(T,T,T,f)=(-1)^{n-1} \partial_ T (f)$. The proof is finished.
\end{proof}
 \begin{pro} Let $T$ and $T^{'}$ be two generalized Reynolds operators on a \textsf{L.t.sRep} pair
$(L, [\cdot,\cdot, \cdot],\theta)$ and $(\phi,\psi)$ be a morphism from $T$ to $T^{'}$. Then  
\begin{itemize}
    \item $\psi$ is a L.t.s morphism from  $(M, [\c,\c, \c]_T)$
to  $(M, [\c,\c, \c]_{T^{'}})$.
\item The following diagram commute for all $u,v\in M$,

$$
 \xymatrix{
  L \ar[d]_{\theta_{T}(u,v)} \ar[r]^{\phi}
                & L \ar[d]^{\theta_{T^{'}}(\psi(u),\psi(v))}  \\
    L \ar[r]_{\phi}
                & L            }
                $$
where  $(L,\theta_T)$ is the induced representation  of the L.t.s $(M,[\c,\c,\c]_T)$ and  $(L,\theta_{T^{'}})$ is the induced representation of the L.t.s $(M,[\c,\c,\c]_T')$.
\end{itemize}
 \end{pro}
$
$
 Motivated by Proposition \ref{prop}, we  show the relevance of our cohomology theory. Let $T$ and $T^{'}$ be two generalized Reynolds operators on $L$ with respect to a representation $(M,\theta)$ and  $(\phi,\psi)$ be a morphism from $T$ to $T^{'}$ such that  $\psi$ is  invertible. Denote by  $C^{2n-1}_T(M,L)$ the space of $(2n-1)$-cochains of a L.t.s $(M,[\c,\c,\c]_T)$  with coefficients in a representation $(L,\theta_T)$. Define
\begin{equation*}
  \begin{cases}
   \Theta :L\otimes L\rightarrow L\otimes L\quad \quad \quad \quad \quad \quad\quad \quad \;\;\;\ if \ n=0,\\
\Theta : C^{2n-1}_{T}(M ,L)\rightarrow C^{2n-1}_{T^{'}}(M ,L)\quad \quad\quad if \ n\geq 1.
\end{cases}
\end{equation*} by
\begin{equation*}
\begin{cases} \Theta(\chi)=\phi(\chi)=(\phi(x_1),\phi(x_2))\quad \quad\quad  \quad \quad\forall\ \chi =(x_1,x_2)\in L\otimes L\quad\;\;\; \ \ if \ n=0,\\
\Theta(\varphi)(u_1,\cdots,u_{2n-1})=\phi(\varphi(\psi^{-1}(u_1),\cdots,\psi^{-1}(u_{2n-1})))\; \ \forall u_i\in M\quad \quad\ if \ n\geq 1.
\end{cases}
\end{equation*}
\begin{thm}
With the above notations, $\Theta$ is a cochain map from the cochain complex $(C^{ \star}_{T}(M ,L),\Lambda_{T})$ to the cochain complex $(C^{\star}_{T^{'}}(M,L),\Lambda_{T^{'}})$. Consequently, it induces a morphism $\overline{\Theta}$ from the cohomology group ${\mathbf{H}}_T^{\star}(M,L)$ to ${\mathbf{H}}_{T'}^{\star}(M,L)$.
\end{thm}
\begin{proof}
For $n=0$, let $\chi\in L\otimes L$ and $(\phi,\psi)$ be a morphism from $T$ to $T^{'}$. Applying  Eqs \eqref{c1}-\eqref{c3}, we obtain \small{
\begin{align*}
    \Theta \circ  \partial_T(\chi)u= &\Phi(\partial_T(\chi)(\psi^{-1}(u))\\=& \phi \circ T\Big(D(\chi)(\psi^{-1}(u)) + \mathcal{H}(\chi,T(\psi^{-1}(u)))\Big)-\phi[\chi,T(\psi^{-1}(u))]\\=
    &T^{'}\circ \psi\Big(D(\chi)(\psi^{-1}(u)) + \mathcal{H}(\chi,T(\psi^{-1}(u)))\Big)-[\phi(\chi),\phi\circ T(\psi^{-1}(u))]
    \\=
    &T^{'} \Big(D(\phi(\chi))\psi(\psi^{-1}(u)) + \mathcal{H}(\phi(\chi),\phi \circ T(\psi^{-1}(u)))\Big)-[\phi(\chi),T^{'}\circ \psi(\psi^{-1}(u))]
     \\=
    &T^{'} \Big(D(\phi(\chi)) u + \mathcal{H}(\phi(\chi),T^{'}u))\Big)-[\phi(\chi),T^{'}u]
     \\=
    &T^{'} \Big(D(\Theta(\chi)) u + \mathcal{H}(\Theta(\chi),T^{'}u))\Big)-[\Theta(\chi),T^{'}u]\\
    =&\partial_T'(\Theta(\chi))u.
\end{align*}}
For $ n \geq 1$, let $\varphi \in  \mathcal{C}_{T}^{2n-1}(M,L)$, we have {\small
\begin{align*}
    &\delta_ T'\Theta(\varphi)(v_1,v_2, \cdots , v_{2n+1})\\ =&\theta_T'(v_{2n},v_{2n+1})\Theta(\varphi)(v_1,v_2, \cdots , v_{2n-1})- \theta_T'(v_{2n-1},v_{2n+1})\Theta(\varphi)(v_1,v_2, \cdots , v_{2n-2},v_{2n}) \\
    &+ \sum_{k=1}^n (-1)^{n+k}D_T'(v_{2k-1},v_{2k})\Theta(\varphi)(v_1,v_2, \cdots , \widehat{v}_{2k-1},\widehat{v}_{2k}, \cdots , v_{2n+1})  \\
    &+ \sum_{k=1}^n \sum_{j=2k+1}^{2n+1}  (-1)^{n+k+1} \Theta(\varphi)(v_1,v_2, \cdots,  \widehat{v}_{2k-1},\widehat{v}_{2k}, \cdots, [v_{2k-1},v_{2k},v_j]_T',\cdots , v_{2n+1})\\
    =&\theta_T'(v_{2n},v_{2n+1})\phi(\varphi(\psi^{-1}(v_1),  \cdots ,\psi^{-1} (v_{2n-1}))- \theta_T'(v_{2n-1},v_{2n+1})\phi(\varphi(\psi^{-1}(v_1), \cdots , \psi^{-1}(v_{2n})) \\
    &+ \sum_{k=1}^n (-1)^{n+k}D_T'(v_{2k-1},v_{2k})\phi(\varphi(\psi^{-1}(v_1), \cdots , \widehat{\psi^{-1}(v_{2k-1})},\widehat{\psi^{-1}(v_{2k})}, \cdots ,\psi^{-1}(v_{2n+1}))  \\
    &+ \sum_{k=1}^n \sum_{j=2k+1}^{2n+1}  (-1)^{n+k+1} \phi(\varphi(\psi^{-1}(v_1), \cdots,   [\psi^{-1}(v_{2k-1}),\psi^{-1}(v_{2k}),\psi^{-1}(v_j)]_T',\cdots , \psi^{-1}(v_{2n+1}))\\
    =&\theta_T'(\psi \circ\psi^{-1}(v_{2n}),\psi \circ\psi^{-1}(v_{2n+1}))\phi(\varphi(\psi^{-1}(v_1),  \cdots ,\psi^{-1} (v_{2n-1}))\\&- \theta_T'(\psi \circ\psi^{-1}(v_{2n-1}),\psi \circ\psi^{-1}(v_{2n+1}))\phi(\varphi(\psi^{-1}(v_1), \cdots , \psi^{-1}(v_{2n})) \\
    &+ \sum_{k=1}^n (-1)^{n+k}D_T'(\psi \circ\psi^{-1}(v_{2k-1}),\psi \circ\psi^{-1}(v_{2k}))\phi(\varphi(\psi^{-1}(v_1), \cdots ,\psi^{-1}(v_{2n+1}))  \\
    &+ \sum_{k=1}^n \sum_{j=2k+1}^{2n+1}  (-1)^{n+k+1} \phi(\varphi(\psi^{-1}(v_1), \cdots,   [\psi^{-1}(v_{2k-1}),\psi^{-1}(v_{2k}),\psi^{-1}(v_j)]_T,\cdots , \psi^{-1}(v_{2n+1}))\\
     =&\phi \Big(\theta_T(\psi^{-1}(v_{2n}),\psi^{-1}(v_{2n+1}))\varphi(\psi^{-1}(v_1),  \cdots ,\psi^{-1} (v_{2n-1}))\\&- \theta_T(\psi^{-1}(v_{2n-1}),\Psi^{-1}(v_{2n+1})) \varphi(\psi^{-1}(v_1), \cdots , \psi^{-1}(v_{2n})) \\
    &+ \sum_{k=1}^n (-1)^{n+k}D_T(\psi^{-1}(v_{2k-1}),\psi^{-1}(v_{2k})) \varphi(\psi^{-1}(v_1), \cdots ,\psi^{-1}(v_{2n+1}))  \\
    &+ \sum_{k=1}^n \sum_{j=2k+1}^{2n+1}  (-1)^{n+k+1}  \varphi(\psi^{-1}(v_1), \cdots,   [\psi^{-1}(v_{2k-1}),\psi^{-1}(v_{2k}),\psi^{-1}(v_j)]_T,\cdots , \psi^{-1}(v_{2n+1}))\Big)\\
    =&\phi \Big(\delta_ T\varphi(\psi^{-1}(v_{1}),\psi^{-1}(v_{2}),\cdots, \psi^{-1}(v_{2n+1}))\Big)\\
     =&\Theta \Big(\delta_ T\varphi\Big)(v_{1} , v_{2},\cdots, v_{2n+1}).
    \end{align*}}
The proof is finished.
\end{proof}
\section{Obstruction class of   a generalized Reynolds operator}\label{Section-5}
In this section, we will use the Yamaguti cohomology theory constructed in the previous section  to investigate  formal deformations of generalized Reynolds operator on a \textsf{L.t.sRep} pair by introducing a special cohomology class associated to an order
$n$ deformation.  We show that a deformation of order $n$ is extendable if and only if this cohomology class in the third
cohomology group is trivial. Thus we call this cohomology class the obstruction class of a deformation of order $n$ being extendable.

Let $\mathbb{K}[[\lambda]]$ be the ring of power series in one variable $\lambda$. For any $\mathbb{K}$-linear space $M$, we denote by  $M[[\lambda]]$ the
vector space of formal power series in $\lambda$ with coefficients in $M$. If in addition, we have a structure of L.t.s $(L, [\cdot,\cdot, \cdot])$ over $\mathbb{K}$, then there is a L.t.s structure over the ring $\mathbb{K}[[\lambda]]$ on $L[[\lambda]]$ given
by
\begin{equation}\label{eq:power3-lie}
\Big[\sum_{i=0}^{+\infty}\lambda^{i}x_i,\sum_{j=0}^{+\infty}\lambda^{j}y_j,\sum_{k=0}^{+\infty}\lambda^{k}z_k\Big]=\sum_{s=0}^{+\infty}\sum_{i+j+k=s}\lambda^{s}[x_i,y_j,z_k],\;\forall x_i,y_j,z_k\in L.
\end{equation}
For any representation $(M,\theta)$ of $(L, [\cdot,\cdot, \cdot])$, there is a natural representation of the L.t.s
$L[[\lambda]]$ on the $\mathbb{K}[[\lambda]]$-module $M[[\lambda]]$, which is given by
\begin{equation}
\theta\Big(\sum_{i=0}^{+\infty}\lambda^{i}x_i,\sum_{j=0}^{+\infty}\lambda^{j}y_j\Big)\Big(\sum_{k=0}^{+\infty}\lambda^{k}m_k\Big)=\sum_{s=0}^{+\infty}\sum_{i+j+k=s}\lambda^{s}\theta(x_i,y_j)m_k,\;\forall x_i,y_j\in L,\;m_k\in M.
\end{equation}
Similarly, the $3$-cocycle $\mathcal{H}$ can be extended to a $3$-cocycle on the L.t.s $L[[\lambda]]$ with coefficients in $M[\lambda]]$,  denoted by the same notation $\mathcal{H}$. Consider a power series
\begin{equation}\label{eq:formal1}
T_\lambda=\sum_{i=0}^{+\infty}\lambda^{i}T_i,\quad\;T_i\in  \mathcal{C}_{T}^{1}(M,L),
\end{equation}
that is $T_\lambda\in  \mathcal{C}_{T}^{1}(M,L[[\lambda]])$. Extend it to be a $\mathbb{K}[[\lambda]]$-module map from
$M[[\lambda]]$ to $L[[\lambda]]$ which is still denoted by $T_\lambda$.
\begin{defi}
If the  power series $T_\lambda=\displaystyle\sum_{i=0}^{+\infty}T_i\lambda^{i}$ with $T_0=T$ satisfies:{\small
\begin{align*}
[T_\lambda u,T_\lambda v,T_\lambda w]=T_\lambda\Big(D(T_\lambda u,T_\lambda v)w+\theta(T_\lambda v,T_\lambda w)u-\theta(T_\lambda u,T_\lambda w)v + \mathcal{H}(T_\lambda u,T_\lambda v,T_\lambda w) \Big),
\end{align*}}
we say that it is a \textbf{formal deformation} of the  generalized Reynolds operator $T$.
\end{defi}
\begin{rmk}

 If $T_\lambda =T + \lambda T_1$ is  a  generalized Reynolds operator on a  \textsf{L.t.sRep} pair $(L, [\cdot,\cdot, \cdot],\theta)$, we say that  $T_1$ \textbf{generates a one-parameter infinitesimal deformation} of  $T$.

\end{rmk}
Based on the relationship between  generalized Reynolds operator and L.t.s structure, we have:
\begin{pro}
Let  $T_\lambda=\displaystyle\sum_{i=0}^{+\infty}\lambda^{i}T_i$ be a formal deformation of a  generalized Reynolds operator $T$ on a \textsf{L.t.sRep} pair
$(L, [\cdot,\cdot, \cdot],\theta)$. Then $[\cdot,\cdot,\cdot]_{T_\lambda}$ defined by
\begin{equation*}
    [u,v,w]_{T_\lambda}
=\displaystyle\sum_{k=0}^{+\infty}\sum_{i+j=k}\lambda^{k}\Big(D(T_i u,T_j v)w+\theta(T_i v,T_j w)u-\theta(T_i u,T_j w)v +\sum_{i+j+s=k} \mathcal{H}(T_i u,T_j v,T_s v)\Big)\quad\forall u,v,w\in M,
\end{equation*}
is a formal deformation of the associated L.t.s $(M,[\cdot,\cdot,\cdot]_T)$ given in Corollary \ref{coro}.
\end{pro}

\begin{defi}
Let $T_\lambda=\displaystyle\sum_{i=0}^{+\infty}\lambda^{i}T_i$ and $T^{'}_\lambda=\displaystyle \sum_{i=0}^{+\infty}\lambda^{i}T{'}_{i}$ be two formal deformations of a generalized Reynolds operator on a \textsf{L.t.sRep} pair
$(L, [\cdot,\cdot, \cdot],\theta)$. They are said to be equivalent if there exists
an element $\chi\in L\wedge L$, $\phi_i\in gl(L)$ and $\psi_i\in gl(M)$ $(i\geq2)$, such that the pair  
\begin{equation}\label{eq:equi}
\Big(\phi_\lambda=Id_L+\lambda[\chi,-]+\sum_{i=2}^{+\infty}\lambda^{i}\phi_i,\quad\;\psi_\lambda=Id_V+\lambda D(\chi)(-)+\lambda \mathcal{H}(\chi,-)+\sum_{i=2}^{+\infty}\lambda^{i}\psi_i\Big)
\end{equation}
is a morphism of  generalized Reynolds operators from  $T_\lambda$ to $T^{'}_\lambda$.
In particular, a formal deformation $T_\lambda$ of a generalized Reynolds operator $T$ is said to be trivial if there
exists an element $\chi\in L\wedge L$, $\phi_i \in  gl(L)$ and $\psi_i \in gl(M)$  ( $i\geq2$) such that $(\phi_\lambda, \psi_\lambda)$ defined by Eq. \eqref{eq:equi}
gives an equivalence between $T_\lambda$ and $T$, with the latter regarded as a deformation of
itself.
\end{defi}
\begin{thm}
If two formal deformations of a generalized Reynolds operator $T$ on a \textsf{L.t.sRep} pair
$(L, [\cdot,\cdot, \cdot],\theta)$ are equivalent, then their infinitesimals
are in the same cohomology class in ${\mathbf{H}}_T^{1}(M,L)$.
\end{thm}
\begin{defi}
Let $T$ be a generalized Reynolds operator on a \textsf{L.t.sRep} pair $(L, [\cdot,\cdot, \cdot],\theta)$. If $T_\lambda=\displaystyle\sum_{i=0}^{n}\lambda^{i}T_i$  defines a $\mathbb{K}[[\lambda]]/ \lambda^{n+1}$-module from $M[[\lambda]]/ \lambda^{n+1} $ to the L.t.s
$L[[\lambda]] / \lambda^{n+1} $
satisfying  
\begin{align*}
   [T_\lambda u,T_\lambda v ,T_\lambda w ]=T_\lambda\Big(D(T_\lambda u ,T_\lambda v )w+\theta(T_\lambda v ,T_\lambda w )u-\theta(T_\lambda u ,T_\lambda w)v + \mathcal{H}(T_\lambda u,T_\lambda v,T_\lambda w) \Big),
\end{align*}
we say that $T_\lambda$ is \textbf{an order $n$ deformation} of the generalized Reynolds operator $T$.
\end{defi}
\begin{defi}
Let $T_\lambda=\displaystyle\sum_{i=0}^{n}\lambda^{i}T_i$ be an order $n$ deformation of a generalized Reynolds operator $T$ on a \textsf{L.t.sRep} pair $(L, [\cdot,\cdot, \cdot],\theta)$. If there
exists a $1$-cochain $T_{n+1}\in  \mathcal{C}_{T}^{1}(M,L) $ such that  $\widetilde{T}_\lambda= T_\lambda + \lambda^{n+1} T_{n+1}$ is an order $(n + 1)$ deformation
of $T$  then we say that $ T_\lambda$ is $\textbf{extendable}$.
\end{defi}
\begin{thm}\label{thm-obs}
  Let $T_\lambda=\displaystyle\sum_{i=0}^{n}\lambda^{i}T_i$ be an order $n$ deformation of $T$. Then $T_\lambda$ is extendable if and only if   the cohomology class $[Obs^T ] \in  {\mathbf{H}}_T^{3}(M,L)$
is trivial, where
$Obs^T $ is a $3$-cochain in $\mathcal{C}_{T}^{3}(M,L)$ defined by  : \small{
\begin{align}
    &Obs^T(u,v,w) \\=& \sum_{0\leq i,j,k\leq n+1\atop
    i+j+k> n+1}  [T_i u,T_j v,T_k w]-T_i\big (D(T_j u,T_k v)w
+\theta (T_j v,T_k w)u-\theta(T_j u,T_k w)v + \mathcal{H}(T_i u,T_j v,T_k w) \Big).\nonumber
\end{align}}
\end{thm}
\begin{proof}
 Let  $\widetilde{T}_\lambda=\displaystyle\sum_{i=0}^{n+1}\lambda^{i}T_i$ be the extension of $T_\lambda$, then for all $u, v, w \in M $
 \begin{align*}
   [\mathcal{\widetilde {T}}_\lambda u,\mathcal{\widetilde {T}}_\lambda v,\mathcal{\widetilde {T}}_\lambda w]=\mathcal{\widetilde {T}}_\lambda\Big(D(\widetilde{T}_\lambda u,\widetilde{T}_\lambda v )w+\theta(\widetilde{T}_\lambda v,\widetilde{T}_\lambda w )u-\theta(\widetilde{T}_\lambda u,\widetilde{T}_\lambda w)v + \mathcal{H}(\widetilde{T}_\lambda u,\widetilde{T}_\lambda v,\widetilde{T}_\lambda w) \Big).
\end{align*}
Expanding the equation and comparing the coefficients of $\lambda^{n}$ yields that:
\begin{align*}
    & [T_{n+1}  u,Tv,Tw]+[ Tu,T_{n+1}v,Tw] +[Tu,T v,T_{n+1} w]- T\big(D(T_{n+1} u,T v)w - \theta(T_{n+1}u,T w)v \\
    & + \mathcal{H}(T_{n+1} u,T v,T w)+D(T u,T_{n+1} v)w + \theta(T_{n+1} v,T w)  v +\mathcal{H}(T u,T_{n+1} v,Tw) \\
    &+\theta(Tv,T_{n+1} w)u - \theta(Tu,T_{n+1} w)v + \mathcal{H}(T u,T v,T_{n+1}w)\big)\\
     &+ \sum_{0\leq i,j,k\leq n+1\atop
    i+j+k> n+1}  [T_i u,T_j v,T_k w]-T_i\big (D(T_j u,T_k v)w
+\theta (T_j v,T_k w)u-\theta(T_j u,T_k w)v + \mathcal{H}(T_i u,T_j v,T_k w) \Big)\\
=&0,
\end{align*}
which is equivalent to  
\begin{align}
    &Obs^T= \delta^{1}_TT_{n+1}.
\end{align}

Thus, the cohomology class $[Obs^T ]$ is trivial. Conversely, suppose that the cohomology class  $Obs^T$ is trivial, then there exists a $1$-cochain $T_{n+1}\in  \mathcal{C}_{T}^{1}(M,L) $ such that $Obs^T= \delta^{1}_T(T_{n+1})$. Set $\widetilde{T}_\lambda= T_\lambda + \lambda^{n+1} T_{n+1}$. Then  $\widetilde{T}_\lambda$ satisfies {\small
\begin{align*}
& \sum_{
    i+j+k=s}  [T_i u,T_j v,T_k w]-T_i\big (D(T_j u,T_k v)w
+\theta (T_j v,T_k w)u-\theta(T_j u,T_k w)v + \mathcal{H}(T_i u,T_j v,T_k w) \Big) \ \forall 0 \leq s\leq n+1,
\end{align*}}
which implies that $\widetilde{T}_\lambda$ is an order $(n + 1)$ deformation of $T$. Hence it is an extension of $T_\lambda$.
\end{proof}
\begin{defi}
 Let $T_\lambda=\displaystyle\sum_{i=0}^{n}\lambda^{i}T_i$ be an order $n$ deformation of $T$. Then
the cohomology class  $[Obs^T ] \in  {\mathbf{H}}_T^{3}(M,L)$ defined in Theorem \ref{thm-obs} is called \textbf{the obstruction class}
of  $T_\lambda$ being extendable.
\end{defi}
\begin{cor}
Let $T$ be a generalized Reynolds operator on a \textsf{L.t.sRep} pair $(L, [\cdot,\cdot, \cdot],\theta)$. If  ${\mathbf{H}}_T^{3}(M,L)=0$,  then every $1$-cocycle in $ \mathcal{Z}^{1}(M,L) $ is the
infinitesimal of some formal deformation of $T$.
\end{cor}
\section{\textsf{NS}-Lie triple systems}\label{Section-6}
The aim of this section is to  introduce the notion of a \textsf{NS}-Lie triple system. A \textsf{NS}-Lie triple system
gives rise to a L.t.s. and a representation on itself. We show that a generalized
Reynolds operator induces a \textsf{NS}-Lie triple system. Thus, \textsf{NS}-Lie triple systems can
be viewed as the underlying algebraic structures of generalized Reynolds operators
on  \textsf{L.t.sRep} pairs.  Also, we show that Lie
algebras, NS-Lie algebras, Lie triple systems and
NS-Lie triple systems are closely related.

\subsection{Definitions and constructions}
\begin{defi}\label{NS-DEf}
Let $L$ be a vector space together with two trilinear maps $\{\cdot,\cdot,\cdot\}$,  $[\cdot,\cdot,\cdot]: \otimes^3 L \to L .$ The triple $(L,\{\cdot,\cdot,\cdot\}, [\cdot,\cdot,\cdot])$ is called \textbf{NS-Lie triple system} if the following identities hold for all $x_1,x_2$ and $y_1,y_2,y_3 \in L$
\begin{align}
&[y_1,y_2,y_3]= - [y_2,y_1,y_3],
    \quad \quad  \circlearrowleft_{y_1,y_2,y_3} [y_1,y_2,y_3]=0,\label{NS-1}\\
    &\{x_1,x_2,[\![ y_1,y_2,y_3]\!]\}=\{\{x_1,x_2,y_1\},y_2,y_3\} -\{\{x_1,x_2,y_2\},y_1,y_3\} +\{y_1,y_2,\{x_1,x_2,y_3\}\}^\ast,\label{NS-2}\\
    &\{x_1,x_2,\{y_1,y_2,y_3\}\}^\ast= \{\{x_1,x_2,y_1\}^\ast,y_2,y_3\} + \{y_1,[\![ x_1,x_2,y_2]\!],y_3\} + \{y_1,y_2,[\![ x_1,x_2,y_3]\!]\},\label{NS-3}\\
    &[x_1,x_2,[\![ y_1,y_2,y_3]\!]]= [[\![x_1,x_2, y_1]\!],y_2,y_3]+ [y_1,[\![x_1,x_2, y_2]\!],y_3]+ [y_1,y_2,[\![x_1,x_2, y_3]\!]]\nonumber\\
    &\quad \quad \quad \quad \quad \quad \quad \quad \quad + \{[x_1,x_2, y_1],y_2,y_3\}- \{[x_1,x_2, y_2],y_1,y_3\} +  \{y_1,y_2,[x_1,x_2, y_3]\}^\ast\nonumber\\
     &\quad \quad \quad \quad \quad \quad \quad \quad \quad-  \{x_1,x_2,[ y_1,y_2,y_3]\}^\ast,\label{NS-4}
\end{align}
where   $\{\cdot,\cdot,\cdot\}^\ast$ and $[\![\cdot,\cdot, \cdot]\!]$ are defined by  
\begin{align}
&\{y_1,y_2,y_3\}^\ast= \{y_3,y_2,y_1\}-\{y_3,y_1,y_2\},\nonumber\\
   & [\![y_1,y_2,y_3]\!]=  \{y_1,y_2,y_3\}^\ast  + \{y_1,y_2,y_3\}-\{y_2,y_1,y_3\} +[y_1,y_2,y_3].\label{crochet-NS}
\end{align}

\end{defi}
\begin{rmk}
Let $(L,\{\cdot,\cdot,\cdot\}, [\cdot,\cdot,\cdot])$ be a \textbf{NS}-Lie triple system. In  the one  hand, if     $\{\cdot,\cdot,\cdot\}$ is trivial, we  get that $(A,[\cdot,\cdot,\cdot])$ is a
L.t.s. In the other hand, if $[\cdot,\cdot,\cdot]$ is trivial, then $(A,\{\cdot,\cdot,\cdot\})$ is  a \textsf{pre-L.t.s}  introduced in  \cite{Mabrouk}. Thus, NS-Lie triple systems
 are     generalizations of both  L.t.s and \textsf{pre-L.t.s}.\\
In the following, we show that \textsf{NS}-Lie triple systems split \textsf{L.t.sRep} pairs.
\end{rmk}
\begin{thm}\label{Ns}
Let $(L,\{\cdot,\cdot,\cdot\}, [\cdot,\cdot,\cdot])$ be a NS-Lie triple system. Then $(L,[\![\cdot,\cdot,\cdot]\!])$ is a L.t.s
which is called the \textbf{sub-adjacent} L.t.s of $(L,\{\cdot,\cdot,\cdot\}, [\cdot,\cdot,\cdot])$ and denoted by $\mathfrak{L}$. Moreover, $(L,\varrho)$ is a representation of $\mathfrak{L}$, where the linear map $\varrho : L \otimes L \to End(L)$ is defined by :
\begin{align*}
    &\varrho(x,y)z = \{z,x,y\}, \quad \quad \forall x,y,z \in L.
\end{align*}
\end{thm}
\begin{proof}It obvious that $[\![\cdot,\cdot,\cdot]\!]$ satisfies \eqref{lts 1}.\\
 By Eqs. \eqref{NS-2}-\eqref{NS-4}, for all $x_1,x_2$ and $ y_1,y_2,y_3 \in L$, we have{\small
\begin{align*}
   &[\![x_1,x_2, [\![y_1,y_2,y_3]\!]]\!]- [\![[\![x_1,x_2, y_1]\!],y_2,y_3]\!]-[\![y_1,[\![x_1,x_2, y_2]\!],y_3]\!]-[\![y_1,y_2,[\![x_1,x_2, y_3]\!]]\!]\\
   =&\{x_1,x_2,\{y_1,y_2,y_3\}\}^\ast+ \{x_1,x_2,\{y_3,y_2,y_1\}\}^\ast- \{x_1,x_2,\{y_3,y_1,y_2\}\}^\ast- \{x_1,x_2,\{y_2,y_1,y_3\}\}^\ast\\
   & - \{x_1,x_2,[y_1,y_2,y_3]\}^\ast- \{x_1,x_2,[\![y_1,y_2,y_3]\!]\} +\{x_2,x_1,[\![y_1,y_2,y_3]\!]\}- [x_1,x_2,[\![y_1,y_2,y_3]\!]]\\
   &-\{y_3,y_2,[\![x_1,x_2,y_1]\!]\}+\{y_3,[\![x_1,x_2,y_1]\!],y_2\}-\{\{x_1,x_2,y_1\}^\ast,y_2,y_3\}-\{\{x_1,x_2,y_1\},y_2,y_3\}\\
   &+ \{\{x_2,x_1,y_1\}^\ast,y_2,y_3\}-\{[x_1,x_2,y_1],y_2,y_3\}+\{y_2,[\![x_1,x_2,y_1]\!],y_3\}- [[\![x_1,x_2,y_1]\!],y_2,y_3]\\
   &-\{y_3,[\![x_1,x_2,y_2]\!],y_1\}+\{y_3,y_1,[\![x_1,x_2,y_2]\!]\}-\{y_1,[\![x_1,x_2,y_2]\!],y_3\} + \{\{x_1,x_2,y_2\}^\ast,y_1,y_3\}\\ &+\{\{x_1,x_2,y_2\},y_1,y_3\}- \{\{x_2,x_1,y_2\},y_1,y_3\}+ \{[x_1,x_2,y_2],y_1,y_3\}- [y_1,[\![x_1,x_2,y_2]\!],y_3]\\
   &-\{\{x_1,x_2,y_3\}^\ast,y_2,y_1\} + \{\{x_1,x_2,y_3\}^\ast,y_1,y_2\}- \{y_1,y_2,\{x_1,x_2,y_3\}\}^\ast+\{y_1,y_2,\{x_2,x_1,y_3\}\}^\ast\\
   &- \{y_1,y_2,[x_1,x_2,y_3]\}^\ast+ \{y_1,y_2,[\![x_1,x_2,y_3]\!]\}- [y_1,y_2,[\![x_1,x_2,y_3]\!]]\\
   =&0,
\end{align*}}
which implies that $\mathfrak{L}$ is a L.t.s. By Eq.\eqref{NS-2}, Eq.\eqref{NS-3}, we have
\begin{align*}
     & \varrho(x_1,x_2)\varrho(y_1,y_2)-\varrho(y_2,x_2)\varrho(y_1,x_1)-\varrho(y_1,[\![y_2,x_1,x_2]\!])+D_\varrho(y_2,x_1)\varrho(y_1,x_2)=0,
    \\
    &\varrho(x_1,x_2)D_\varrho(y_1,y_2)-D^\ast(y_1,y_2)\varrho(x_1,x_2)+\varrho([\![y_1,y_2,x_1]\!],x_2)+\varrho(x_1,[\![y_1,y_2,x_2]\!])=0,
\end{align*}
where $D_\varrho(x_1,x_2)=\varrho(x_2,x_1)- \varrho(x_1,x_2)$. Therefore, $ \varrho$ is a representation of $\mathfrak{L}$ on $L$.  Then $(L,[\![\cdot,\cdot,\cdot]\!],\varrho)$ is a \textsf{L.t.sRep} pair.
\end{proof}
The following results illustrate that \textsf{NS}-Lie triple systems can be viewed as the underlying algebraic
structures of generalized Reynolds operators on   \textsf{L.t.sRep} pairs.
\begin{thm}\label{thm-NS}
 Let $T:M \to L $ be a generalized Reynolds operator on a \textsf{L.t.sRep} pair $(L, [\cdot,\cdot, \cdot],\theta)$. Then
\begin{align}\label{induced-NS}
    &\{u,v,w\}= \theta(Tv,Tw)u \quad\text{and} \quad
    [u,v,w]=\mathcal{H}(Tu,Tv,Tw)
\end{align}
defines a  \textsf{NS-Lie triple system} structure on $M$. It is called\textbf{ the induced \textbf{NS-Lie triple system}}.
\end{thm}
\begin{proof}
For any $u,v,w \in L $, it is obvious that
\begin{align*}
    &[u,v,w]= - [v,u,w]
    \quad and \quad  \circlearrowleft_{u,v,w} [u,v,w]=0.
\end{align*}
Furthermore, for  all $u_1, u_2, v_1, v_2, v_3 \in M$, by Eqs. \eqref{rep lts 1}-\eqref{rep lts 2} and Eq. \eqref{crochet-NS}, we have{\small
\begin{align*}
    &\{u_1,u_2,[\![ v_1,v_2,v_3]\!]\}-\{\{u_1,u_2,v_1\},v_2,v_3\} +\{\{u_1,u_2,v_2\},v_1,v_3\} -\{v_1,v_2,\{u_1,u_2,v_3\}\}\\
    =&\{u_1,u_2,\{ v_1,v_2,v_3\}^\ast\} + \{u_1,u_2,\{ v_1,v_2,v_3\}\}- \{u_1,u_2,\{ v_2,v_1,v_3\}\}+\{u_1,u_2,[v_1,v_2,v_3]\}\\
&-\{\{u_1,u_2,,v_1\},v_2,v_3\} +\{\{u_1,u_2,v_2\},v_1,v_3\} -\{v_1,v_2,\{u_1,u_2,v_3\}\}^\ast\\
=&\{u_1,u_2,\{ v_3,v_2,v_1\}\}- \{u_1,u_2,\{ v_3,v_1,v_2\}\}+ \{u_1,u_2,\{ v_1,v_2,v_3\}\}- \{u_1,u_2,\{ v_2,v_1,v_3\}\}\\
&+\{u_1,u_2,[v_1,v_2,v_3]\}-\{\{u_1,u_2,v_1\},v_2,v_3\} +\{\{u_1,u_2,v_2\},v_1,v_3\} -\{\{u_1,u_2,v_3\},v_2,v_1\}\\
&+ \{\{u_1,u_2,v_3\},v_1,v_2\}\\
=&\theta(Tu_2,[Tv_1,Tv_2,Tv_3])u_1 - \theta(Tv_2,Tv_3)\theta(Tu_2,Tv_1)u_1+ \theta(Tv_1,Tv_3)\theta(Tu_2,Tv_2)u_1\\&-\theta(Tv_2,Tv_1)\theta(Tu_2,Tv_3)u_1+\theta(Tv_1,Tv_2)\theta(Tu_2,Tv_3)u_1=0.
\end{align*}}
This implies that Eq.\eqref{NS-2} in Definition \ref{NS-DEf} holds. Similarly for  Eq.\eqref{NS-3}. Moreover, since $\mathcal{H}$ is a $3$-cocycle of $L$ with coefficients in $M$, then Eq.\eqref{NS-4} holds. Therfore, $ (L, \{\cdot,\cdot,\cdot\}, [\cdot,\cdot,\cdot])$ is a NS-Lie triple system.
\end{proof}

\begin{cor}
The subadjacent L.t.s $ (M,[\![\cdot,\cdot,\cdot]\!])$ of the above \textsf{NS-L.t.s} $(M,\{\cdot,\cdot,\cdot\}, [\cdot,\cdot,\cdot])$ is exactly
the L.t.s given in Corollary \ref{coro}.
\end{cor}
\begin{defi}
A  morphism from a \textsf{NS-L.t.s} $(L,\{\cdot,\cdot,\cdot\}, [\cdot,\cdot,\cdot])$ to $(L',\{\cdot,\cdot,\cdot\}', [\cdot,\cdot,\cdot]')$
is a linear map $ \psi: L \to L'$ satisfying 
\begin{align*}
    & \psi \{x,y,z\}= \{ \psi( x),\psi(y),\psi( z)\}'\quad and \quad   \psi [x,y,z]= [ \psi(x),\psi(y),\psi (z)]' \quad \forall x,y,z \in L.
\end{align*}
\end{defi}
\begin{pro}
Let $T$ and $T'$ be two generalized Reynolds operators on a \textsf{L.t.sRep} pair   $(L, [\cdot,\cdot, \cdot],\theta)$, and $(L,\{\cdot,\cdot,\cdot\}, [\cdot,\cdot,\cdot])$, $(L',\{\cdot,\cdot,\cdot\}', [\cdot,\cdot,\cdot]')$ be the induced \textsf{NS-L.t.s}. Let $(\phi,\psi)$ be a  morphism from $T$ to $T'$, then $\psi$ is a morphism from the
\textsf{NS-L.t.s} $(L,\{\cdot,\cdot,\cdot\}, [\cdot,\cdot,\cdot])$ to $(L',\{\cdot,\cdot,\cdot\}', [\cdot,\cdot,\cdot]')$.
\end{pro}
\begin{proof}
It is straightforward  by Eq.\eqref{c1}-\eqref{c3} and Eq.\eqref{induced-NS}.
\end{proof}

Since $\lambda$-weighted Reynolds operators L.t.s is a particular generalized Reynolds operator on a \textsf{L.t.sRep} pair, then  the following corollary is obvious.
\begin{cor}\label{cor}
Let $(L,[\c,\c,\c],\mathcal{R} )$ be a $\lambda $-weighted Reynolds operator L.t.s. Then  $(L,\{\cdot,\cdot,\cdot\},\lambda[\c,\c,\c]\circ(\mathcal{R}\otimes\mathcal{R}\otimes\mathcal{R})  )$ is a \textsf{NS-L.t.s}, where
\begin{align*}
    &\{x,y,z\}= [x,\mathcal{R}y,\mathcal{R}z] \ \forall x,y,z\in L.
\end{align*}
\end{cor}
\begin{ex}
 Example \ref{Expl} can be enhanced to NS-Lie triple system. More precisely, we make a $\lambda$-weighted Reynolds operator on a Lie algebra $\mathcal{R}$ to construct  a generalized Reynolds operator on the induced   L.t.s $(L,[[\c,\c],\c]  )$. Moreover by Corollary \ref{cor}, we obtain   a \textsf{NS-L.t.s} on $W_{\geq0}$ defined by the structures 
\begin{align*}
    \{l_m,l_n,l_p\}&= [l_m,\mathcal{R}(l_n),\mathcal{R}(l_p)]=\frac{ 1}{\lambda^2(n+1)(p+1) }[l_m,l_n,l_p]\\
    &=\frac{(m-n)(m+n-p))}{\lambda^2(n+1)(p+1) }l_{m+n+p},\\
    2\lambda[\mathcal{R}(l_m),\mathcal{R}(l_n),\mathcal{R}(l_p)]&=2\frac{(m-n)(m+n-p))}{\lambda^2 (m+1)(n+1)(p+1)}l_{m+n+p} \quad \forall m,n,p\in \mathbb{Z}_+. \end{align*} 
    So $\mathcal{R}$  is  a generalized Reynolds operator on the induced   L.t.s $(L,[\c,\c,\c]  )$. 
    \end{ex}
    \begin{ex}More generally, consider the $\lambda $-weighted Reynolds operator
       $\mathcal{R}:B(q)_{\geq0} \to B(q)_{\geq0} $   on the L.t.s $(B(q)_{\geq0},[\cdot,\cdot,\cdot])$ defined in  Example \ref{Block}. Its follows from  Corollary \ref{cor}, that  $B(q)_{\geq0}$ carries a \textsf{NS}-L.t.s  structure $(B(q)_{\geq0},\{\cdot,\cdot,\cdot\},\lambda[\cdot,\cdot,\cdot])$
   given by
 \begin{align*}
&\{L_{m,i},L_{n,j},L_{p,k}\}= \frac{(n(i + q) - m(j + q)) (p(i + j+q) - (m+n)(p + q))}{ \lambda^2(n+j + 1) (p+k + 1)}L_{m+n+p,i+j+k}.
\end{align*}

 \end{ex}
   \subsection{From \textsf{NS}-Lie algebras to \textsf{NS}-Lie triple systems}
Now, we recall the definition of \textsf{NS}-Lie algebra introduced by A. Das in \cite{Das-Twisted}, so that we can construct a  \textsf{NS}-Lie triple system starting from a  \textsf{NS}-Lie algebra.
\begin{defi}
A \textsf{NS}-Lie algebra is a vector space L together with bilinear operations $\circ,\curlyvee  : L \to L$ in which $\curlyvee$ is skew-symmetric
and satisfying the following two identities:
\begin{align}\label{NS-Lie-1}
    & (x \ast y) \circ z - x\circ ( y \circ z)  +  y\circ( x  \circ z) =0,\\\label{NS-Lie-2}
    & x \curlyvee( y \ast z) + y \curlyvee ( z \ast x) + z \curlyvee ( x \ast z)  +  x\circ( y  \curlyvee z) + y  \circ (z\curlyvee x)+ z \circ (x\curlyvee y) =0  \quad \forall x,y,z \in L,
\end{align}
 where   $ x \ast y =  x \circ y - y\circ x + x\curlyvee y.$
\end{defi}

\begin{pro}
Let $(L,\circ,\curlyvee )$ be a \textsf{NS}-Lie algebra. Then $(L,\ast)$ is a Lie algebra  called the \textbf{adjacent} Lie algebra.
\end{pro}

Let $(L,[\c,\c] , \rho)$ be a \textsf{LieRep} pair. A skew-symmetric bilinear map  $H:L \times L  \to M $ is a $2$-cocycle if it  satisfies  
\begin{align}
    &\circlearrowleft_{x,y,z}\big(\rho(x) H(y,z) + H(x,[y,z])\big) =0  \ \ \ \  \forall x,y,z \in L.
\end{align}
A linear map $T  : M \to L$ is said to be a \textbf{generalized Reynolds operator} if it satisfies :
\begin{align}
   [T(u),T(v)]= T \Big( \rho(Tu)v - \rho(Tv)u + H(Tu,Tv)\Big)  \ \ \ \ \ \forall u,v \in M.
\end{align}
Define  two multiplications: \begin{equation}
    u\circ v=\rho(Tu)v \quad \text{and} \quad u\curlyvee v=H(Tu,Tv) \ \ \ \ \ \forall u,v \in M.
\end{equation}
\begin{pro}
Under the above notations, $(M,\circ,\curlyvee )$  is a \textsf{NS}-Lie algebra.
\end{pro}
\begin{thm}\label{Thm-construction} Let
  $(L,\circ ,\curlyvee)$ be a \textsf{NS}-Lie algebra. Then $(L,\{\cdot,\cdot,\cdot\},[\cdot,\cdot,\cdot])$ is a  \textsf{NS}-Lie  triple system, where
 \begin{align}
     &\{x,y,z\} = z \circ (y \circ x),\\
     &[x,y,z]=  (x \ast y) \curlyvee z - z \circ (x\curlyvee y ) \ \ \forall x,y,z \in L.
 \end{align}
Note that $[\![x,y,z]\!]=(x\ast y)\ast z$.
\end{thm}

\begin{proof}
Let $x_1,x_2$ and $y_1,y_2,y_3 \in L$. It is obvious that the bracket $[\cdot,\cdot,\cdot]$ is skew-symmetric and $ \circlearrowleft_{y_1,y_2,y_3} [y_1,y_2,y_3]=0$. According to \eqref{NS-Lie-1},  we have
\begin{align*}
    & \{\{x_1,x_2,y_1\},y_2,y_3\} -\{\{x_1,x_2,y_2\},y_1,y_3\} +\{y_1,y_2,\{x_1,x_2,y_3\}\}^\ast- \{x_1,x_2,[\![ y_1,y_2,y_3]\!]\}\\=& \underline{\underline{{ y_3 \circ \Big( y_2 \circ (y_1 \circ (x_2\circ x_1))  \Big) - y_3 \circ \Big( y_1 \circ (y_2 \circ (x_2\circ x_1))  \Big)}}} + \underline{ y_1 \circ \Big( y_2 \circ (y_3 \circ (x_2\circ x_1))}  \Big)\\
    &\underline{-  y_2 \circ \Big( y_1 \circ (y_3 \circ (x_2\circ x_1))  \Big)}- \Big(( y_1 \ast  y_2)  \ast y_3\Big)  \circ (x_2\circ x_1) \\
   \overset{\eqref{NS-Lie-1}}{=}& (y_1 \ast y_2) \circ  \Big(y_3 \circ (x_2\circ x_1)\Big) -   y_3 \circ \Big(  (y_1\ast y_2) \circ (x_2\circ x_1))   \Big) - \Big(( y_1 \ast  y_2)  \ast y_3\Big)  \circ (x_2\circ x_1)\\
  =&0.
\end{align*}
Then Eq.\eqref{NS-2} holds. Similarly
\begin{align*}
   & \{\{y_1,y_2,y_3\},x_1,x_2\} - \{\{y_1,y_2,y_3\},x_2,x_1\} + \{\{y_1,x_2,x_1\},y_2,y_3\} -\{\{y_1,x_1,x_2\},y_2,y_3\}\\
   & + \{y_1,[x_1,x_2,y_2],y_3\}+  \{y_1,y_2,[x_1,x_2,y_3]\}\\
   =& \underline{\underline{x_2 \circ \Big( x_1 \circ (y_3 \circ (y_2 \circ y_1)) \Big) - x_1 \circ \Big( x_2 \circ (y_3 \circ (y_2 \circ y_1)) \Big) }}+\underline{ y_3 \circ \Big( y_2 \circ (x_1 \circ (x_2,y_1)) }\Big)  \\
   &\underline{-  y_3 \circ \Big( y_2 \circ (x_2 \circ (x_1,y_1)) \Big) } + y_3 \circ \Big( ((x_1 \ast x_2)\ast y_2) \circ y_1 \Big) +  \Big( (x_1 \ast x_2)\ast y_3\Big) \circ (y_2 \circ y_1) \\
    \overset{\eqref{NS-Lie-1}}{=} &  y_3 \circ (y_2 \circ \Big( (x_1 \ast x_2) \circ y_1 \Big)- (x_1 \ast x_2) \circ \Big(y_3 \circ (y_2 \circ y_1)\Big)+ y_3 \circ \Big( ((x_1 \ast x_2)\ast y_2) \circ y_1 \Big) \\&+  \Big( (x_1 \ast x_2)\ast y_3\Big) \circ (y_2 \circ y_1)\\
    =& y_3 \circ \Big( (x_1 \ast x_2) \circ (y_2 \circ y_1) \Big) - y_3 \circ \Big( (x_1 \ast x_2) \circ (y_2 \circ y_1) \Big) = 0.
\end{align*}
Therefore Eq.\eqref{NS-3} holds. Also, we have
\begin{align*}
    &[[\![x_1,x_2, y_1]\!],y_2,y_3]+ [y_1,[\![x_1,x_2, y_2]\!],y_3]+ [y_1,y_2,[\![x_1,x_2, y_3]\!]] + \{[x_1,x_2, y_1],y_2,y_3\}\\&- \{[x_1,x_2, y_2],y_1,y_3\} +  \{y_1,y_2,[x_1,x_2, y_3]\}^\ast-  \{x_1,x_2,[ y_1,y_2,y_3]\}^\ast- [x_1,x_2,[\![ y_1,y_2,y_3]\!]]\\
    =& \Big (((x_1 \ast x_2)\ast y_1 )\ast y_2\Big) \curlyvee  y_3 + \cancel{
  \Big ((y_1 \ast y_2)\ast y_3\Big) \circ (x_1\curlyvee  x_2)}- y_3 \circ \Big(  ((x_1 \ast x_2)\ast y_1) \curlyvee  y_2\Big)\\
  & + \Big (y_1 \ast ((x_1 \ast x_2)\ast y_2)\Big ) \curlyvee  y_3 -
  y_3 \circ \Big( y_1   \curlyvee ((x_1 \ast x_2)\ast y_2)     \Big ) + (y_1 \ast y_2)
  \curlyvee \Big(  (x_1 \ast x_2)\ast y_3 \Big )\\
 & -  \Big(  (x_1 \ast x_2)\ast y_3 \Big )  \circ (y_1 \curlyvee y_2) +  y_3 \circ \Big( y_2  \circ ((x_1\ast  x_2) \curlyvee y_1)- \cancel{ y_3 \circ \Big( y_2  \circ ( y_1   \circ (x_1\curlyvee  x_2))\Big)  }\\
 & -  y_3 \circ \Big( y_1  \circ ((x_1\ast  x_2) \curlyvee y_2) + y_1 \circ \Big( y_2\circ (
 (x_1 \ast x_2)\curlyvee  y_3 ) \Big) +  \cancel{ y_3 \circ \Big( y_1  \circ ( y_2   \circ (x_1\curlyvee  x_2))\Big)  } \\& -  y_2  \circ\Big(  y_1 \circ( (x_1 \ast x_2) \curlyvee y_3)\Big)   - \cancel{ y_1 \circ \Big( y_2  \circ ( y_3   \circ (x_1\curlyvee  x_2)\Big)  } + \cancel{ y_2 \circ \Big( y_1  \circ ( y_3   \circ (x_1\curlyvee  x_2)\Big)  }\\
 & -     x_1 \circ \Big( x_2  \circ (y_1 \ast y_2) \curlyvee y_3 \Big) +   x_1 \circ \Big( x_2  \circ (y_3  \circ (y_1 \curlyvee y_2)) \Big) +    x_2 \circ \Big( x_1  \circ (y_1 \ast y_2) \curlyvee y_3 \Big)\\
 & -   x_2 \circ \Big( x_1  \circ (y_3  \circ (y_1 \curlyvee y_2)) \Big)- (x_1 \ast x_2) \curlyvee \Big ( (y_1 \ast y_2) \ast y_3\Big )\\
  \overset{\eqref{NS-Lie-1}}{=} &\underline{ \Big (((x_1 \ast x_2)\ast y_1 )\ast y_2\Big) \curlyvee  y_3} - y_3 \circ \Big(  ((x_1 \ast x_2)\ast y_1) \curlyvee  y_2\Big) +\underline{ \Big (y_1 \ast ((x_1 \ast x_2)\ast y_2)\Big ) \curlyvee  y_3}\\
 & - y_3 \circ \Big( y_1   \curlyvee ((x_1 \ast x_2)\ast y_2)     \Big ) + (y_1 \ast y_2)
  \curlyvee \Big(  (x_1 \ast x_2)\ast y_3 \Big )+  y_3 \circ \Big( y_2  \circ ((x_1\ast  x_2) \curlyvee y_1) \\
  &  -  y_3 \circ \Big( y_1  \circ ((x_1\ast  x_2) \curlyvee y_2) + y_1 \circ \Big( y_2\circ (
 (x_1 \ast x_2)\curlyvee  y_3 ) \Big) -  y_2  \circ\Big(  y_1 \circ( (x_1 \ast x_2) \curlyvee y_3)\Big)\\
 & -   x_1 \circ \Big( x_2  \circ (y_1 \ast y_2) \curlyvee y_3 \Big) +  x_2 \circ \Big( x_1  \circ (y_1 \ast y_2) \curlyvee y_3 \Big) \underline{\underline{+  x_1 \circ \Big( x_2  \circ (y_3  \circ (y_1 \curlyvee y_2)) \Big)}}\\
 &  \underline{\underline{- x_2 \circ \Big( x_1  \circ (y_3  \circ (y_1 \curlyvee y_2)) \Big)}}- (x_1 \ast x_2) \curlyvee \Big ( (y_1 \ast y_2) \ast y_3\Big )\underline{\underline{-\Big( (x_1 \ast x_2)\ast y_3\Big)\circ (y_1 \curlyvee y_2)}}\\
 \overset{\eqref{NS-Lie-1}}{=} &  \Big ((x_1 \ast x_2)\ast( y_1 \ast y_2)\Big) \curlyvee  y_3 + y_3 \circ \Big(  ((x_1 \ast x_2)\circ ( y_1 \curlyvee  y_2\Big) -  y_3 \circ \Big((x_1 \ast x_2)\ast y_1)\curlyvee y_2 \Big ) \\
 &- y_3 \circ \Big(  y_1 \curlyvee ( (x_1 \ast x_2)\ast y_2)\Big) +  (y_1 \ast y_2) \curlyvee \Big((x_1 \ast x_2)\ast y_3 \Big) +  y_3 \circ \Big( y_2 \circ ((x_1 \ast x_2) \curlyvee  y_1 ) \Big)\\
 & - y_3 \circ \Big( y_1 \circ ((x_1 \ast x_2) \curlyvee  y_2 ) \Big)\underline{\underline{+ y_1 \circ \Big( y_2 \circ ((x_1 \ast x_2) \curlyvee  y_3 ) \Big) -y_2 \circ \Big( y_1 \circ ((x_1 \ast x_2) \curlyvee  y_3 ) \Big)}}\\
 & -  \underline{ x_1 \circ \Big( x_2  \circ (y_1 \ast y_2) \curlyvee y_3 \Big) +  x_2 \circ \Big( x_1  \circ (y_1 \ast y_2) \curlyvee y_3 \Big)}- (x_1 \ast x_2) \curlyvee \Big ( (y_1 \ast y_2) \ast y_3\Big )\\
 \overset{\eqref{NS-Lie-2}}{=} &  \Big ((x_1 \ast x_2)\ast( y_1 \ast y_2)\Big) \curlyvee  y_3 +\underline{ y_3 \circ \Big(  ((x_1 \ast x_2)\circ ( y_1 \curlyvee  y_2\Big) -  y_3 \circ \Big((x_1 \ast x_2)\ast y_1)\curlyvee y_2 \Big )} \\
 &- \underline{y_3 \circ \Big(  y_1 \curlyvee ( (x_1 \ast x_2)\ast y_2)\Big) }+  (y_1 \ast y_2) \curlyvee \Big((x_1 \ast x_2)\ast y_3 \Big) + \underline{ y_3 \circ \Big( y_2 \circ ((x_1 \ast x_2) \curlyvee  y_1 ) \Big)}\\
 & -\underline{ y_3 \circ \Big( y_1 \circ ((x_1 \ast x_2) \curlyvee  y_2 ) \Big) }-
 (x_1 \ast x_2) \circ \Big( (y_1 \ast y_2) \curlyvee y_3\Big)+ (y_1 \ast y_2) \circ \Big( (x_1 \ast x_2) \curlyvee y_3\Big)\\
 &- (x_1 \ast x_2) \curlyvee \Big ( (y_1 \ast y_2) \ast y_3\Big )\\
  \overset{\eqref{NS-Lie-2}}{=}  & - \Big ( (x_1 \ast x_2) \curlyvee \Big ( (y_1 \ast y_2) \ast y_3\Big ) +  (y_1 \ast y_2) \curlyvee \Big( y_3\ast(x_1 \ast x_2) \Big)
  +   y_3\curlyvee\Big ((x_1 \ast x_2)\ast( y_1 \ast y_2)\Big)\\
  &+ (x_1 \ast x_2)  \circ \Big( (y_1 \ast y_2) \curlyvee y_3\Big)
  + ( y_1 \ast y_2) \circ \Big( y_3 \curlyvee  (x_1 \ast x_2) \Big)
  +  y_3 \circ \Big(  (x_1 \ast x_2)  \curlyvee  (y_1 \ast y_2) \Big) \Big)\\
  &=0.
    \end{align*}
    Then \eqref{NS-4} holds and $(L,\{\cdot,\cdot,\cdot\},[\cdot,\cdot,\cdot])$ is a  \textsf{NS}-Lie  triple system.
\end{proof}

 We have shown  that  Lie
algebras, NS-Lie algebras, Lie triple systems and
NS-Lie triple systems are closely related  in the sense of commutative diagram of categories as follows :
\begin{equation}\label{Diagramme1}
\begin{split}
{\xymatrix{
\ar[rr] \mbox{\bf   NS-Lie alg.$(L ,\circ,\curlyvee)$ }\ar[d]_{\mbox{}}\ar[rr]^{\mbox{}}
                && \mbox{\bf NS-L.t.s $(L,\{\cdot,\cdot,\cdot\} ,[\cdot,\cdot,\cdot])$ }\ar[d]_{\mbox{ }}\\
\ar[rr] \mbox{\bf LieRep pair. $(L,[\cdot,\cdot],\rho)$}\ar@<-1ex>[u]_{\mbox{}}\ar[rr]^{\mbox{ }}
                && \mbox{\bf L.t.sRep pair  $(L,[[\cdot,\cdot],\cdot],\theta_\rho).$}\ar@<-1ex>[u]_{\mbox{}}}
}\end{split}
\end{equation}


\begin{thebibliography}{999}

\bibitem{Block} R. Block, \emph{On torsion-free abelian groups and Lie algebras}. Proceedings of the American Mathematical Society. 
613--620. (1958).




\bibitem{Cartan}E. Cartan. \emph{Oeuvres completes.}  Part 1. Gauthier-Villars. Paris. vol. 2. nos. 101--138.  (1952).










\bibitem{chtioui}
 T. Chtioui, A. Hajjaji, S. Mabrouk and A. Makhlouf,  \emph{Cohomology and deformations of twisted
$\mathcal O$-operators on 3-Lie algebras}, Filomat, 37(21) (2023), 6977-6994.

\bibitem{O-operator}
 T. Chtioui, A. Hajjaji, S. Mabrouk and  A. Makhlouf, \emph{Cohomologies and deformations of $\mathcal O$-operators
on Lie triple systems}. Journal of Mathematical Physics, 64(8) (2023), 081701. 

\bibitem{CK}
 A. Connes and  D. Kreimer,  \emph{Renormalization in quantum field theory and the Riemann-Hilbert problem.   The
Hopf algebra structure of graphs and the main theorem.} Comm. Math. Phys. 210. 249--273. (2000).


\bibitem{Das-Twisted}
   A. Das,  \emph{Twisted Rota–Baxter operators and Reynolds operators on Lie algebras and NS-Lie algebras.} Journal of Mathematical Physics. 62(9). 091701. (2021).

\bibitem{Das}
   A. Das, \emph{Deformations of associative Rota-Baxter operators.}  J. Algebra.560 . 144--180. (2020).

\bibitem{Das-1}    A. Das,  \emph{Twisted Rota–Baxter operators and Reynolds operators on Lie algebras and NS-Lie algebras. Journal of Mathematical Physics.} 62(9). 091701. (2021).

\bibitem{Das-2}   A. Das,  \emph{Cohomology and deformations of twisted Rota-Baxter operators and NS-algebras.}   ArXiv preprint-2010.01156.
(2020).

\bibitem{Getzler}
 E. Getzler, \emph{Lie theory for nilpotent $L_{\infty}$-algebras.} Ann. of Math.  271--301. (2009).







\bibitem{Gub-AMS}  L. Guo,  \emph{What is a Rota-Baxter algebra Notice}, Amer.Math. Soc. 1436-1437.(2009). 

\bibitem{Gub}  L. Guo, \emph{An introduction to Rota-Baxter algebra}. International Pres, xii+226. (2012).

\bibitem{Harris}
 B. Harris, \emph{Cohomology of Lie triple systems and Lie algebras with involution.} Trans. Amer.
Math. Soc. 98. 148--162. (1961).

\bibitem{Hodge}
 O. L. Hodge and  B. J. Parshall, \emph{On the representation theory of Lie triple systems.  Trans.}
Amer. Math. Soc.354.  4359--4391.  (2002).

\bibitem{Jacobson}
  N. Jacobson, \emph{ Lie and Jordan triple systems.}  Amer. J. Math. 71 149--170. (1949).

\bibitem{KAM} J.  Kamp\'{e} de F\'{e}riet,  \emph{Introduction to the statistical theory of turbulence.  Correlation and spectrum. The Institute for Fluid Dynamics and Applied Mathematics.} University of Maryland, College Park. Md.  iv+162 pp.(1951).

\bibitem{Kubo}
 F. Kubo and Y. Taniguchi, \emph{A controlling cohomology of the deformation theory of Lie triple
systems.} J. Algebra 278,  242--250 (2004).







\bibitem{Ku}
  B. A. Kupershmidt, \emph{What a classical $r$-matrix really is.} J. Nonlinear Math. Phys. 448--488. (1999).

\bibitem{LG} P. Lei and L. Guo. \emph{Nijenhuis algebras, NS algebras and N-dendriform algebras. Front. Math. 827--846. (2012).}

\bibitem{Leroux} P. Leroux, \emph{ Construction of Nijenhuis operators and dendriform trialgebras}. Int. J. Math. Math. Sci.  no. 49--52, 2595--2615.(2004).

\bibitem{Lister}
W. G. Lister, \emph{ A structure theory of Lie triple systems.} Trans. Amer. Math. Soc.72. 217--242.(1952).



\bibitem{Mabrouk}
  S. Mabrouk. \emph{Pre-Lie triple system structures and
generalized derivations}. preprint. (2021).

\bibitem{Re}   O. Reynolds, \emph{On the dynamical theory of incompressible viscous fluids and the determination of the criterion.} Phil. Trans. Roy. Soc.  A 136. 123--164. (1895). Reprinted in Proc. Roy. Soc. London Ser. A 451. no. 1941, 5--47. (1995).





\bibitem{PBG}    J. Pei, C. Bai and   L. Guo, \emph{Splitting of operads and Rota-Baxter operators on operads.}  Appl. Cate. Stru. 25. 505--538. (2017).




\bibitem{STS} M. A. Semonov-Tian-Shansky, \emph{What is a classical R-matrix Funct.} Anal. Appl. 17. 259--272. (1983).

\bibitem{Sheng}
 H. Shuai and Y. Sheng, \emph{Generalized Reynolds operators on 3-Lie algebras and NS-3-Lie algebras.}  International Journal of Geometric Methods in Modern Physics 18.14.  2150223. (2021).


\bibitem{TBGS}  R. Tang, C. Bai, L. Guo and Y. Sheng, \emph{Deformations and their controlling cohomologies of $\mathcal O$-operators.}  Comm. Math. Phys. 368. no. 2. 665--700. (2019).


\bibitem{Uch}
K. Uchino, \emph{Twisting on associative algebras and Rota-Baxter type operators.}  J. Noncommut. Geom.  4  349--379. (2010).
















\bibitem{Yamaguti}
K. Yamaguti, \emph{On the cohomology space of Lie triple systems} Kumamoto J. Sci. A. 5. 
44--52. (1960).

\bibitem{Zhang}
T. Zhang, \emph{Notes on Cohomologies of Lie Triple Systems.} J. Lie Theory. 24(4). 909--929. (2014).








\bibitem{gao-guo}  T. Zhang, X. Gao and  L. Guo. \emph{Reynolds algebras and their free objects from bracketed words and rooted trees.} Journal of Pure and Applied Algebra. 225(12). 106766. (2021).

\end{thebibliography}
\end{document}